\documentclass[12pt,draft,onecolumn]{IEEEtran}

\usepackage{pdfsync}

\usepackage{amsfonts}
\usepackage{amssymb}
\usepackage{amsmath}
\usepackage{color}
\usepackage{cancel}
\usepackage{pstricks}
\usepackage{psfrag}
\usepackage{mathrsfs}
\usepackage{graphicx}
\def\be{\begin{equation}}
\def\ee{\end{equation}}
\def\bea{\begin{eqnarray}}
\def\eea{\end{eqnarray}}
\def\beann{\begin{eqnarray*}}
\def\eeann{\end{eqnarray*}}

\newenvironment{smallarray}[1]
 {\null\,\vcenter\bgroup\scriptsize
  \arraycolsep=.42em
  \hbox\bgroup$\array{@{}#1@{}}}
 {\endarray$\egroup\egroup\,\null}

\newtheorem{lemma}{Lemma}
\newtheorem{theorem}{Theorem}
\newtheorem{remark}{Remark}
\newtheorem{corollary}{Corollary}

\newtheorem{definition}{Definition}
\newtheorem{problem}{Problem}

\newtheorem{conjecture}{Conjecture}

\def\be{\begin{equation}}
\def\ee{\end{equation}}
\def\bea{\begin{eqnarray}}
\def\eea{\end{eqnarray}}
\def\beann{\begin{eqnarray*}}
\def\eeann{\end{eqnarray*}}

\def\QED{\mbox{\rule[0pt]{1.5ex}{1.5ex}}}
\def\endproof{\hspace*{\fill}~\QED\par\endtrivlist\unskip}

\def\R{\mathbb{R}}
\def\C{\mathbb{C}}
\def\Z{\mathbb{Z}}
\def\rk{\mathrm{rk}}
\def\mindeg{\min\,\deg\,}
\def\maxdeg{\max\,\deg\,}

\newcommand{\diag}{\operatorname{diag}}

\definecolor{Royalblue}{cmyk}{1,0.30,0.2,0.2}

\definecolor{Darkgreen}{rgb}{0,0.6,0}

\begin{document}
\title{{On the Factorization of Rational Discrete-Time Spectral Densities}}

\author{Giacomo Baggio, Augusto~Ferrante
\thanks{Giacomo Baggio is with the Dipartimento di Ingegneria dell'Informazione,
         Universit\`a di Padova,
         via Gradenigo, 6/B -- I-35131 Padova, Italy.
        E-mail: {\tt  giacomo.baggio@studenti.unipd.it.} }
\thanks{Augusto Ferrante is with the Dipartimento di Ingegneria dell'Informazione,
         Universit\`a di Padova,
         via Gradenigo, 6/B -- I-35131 Padova, Italy.
        E-mail: {\tt augusto@dei.unipd.it.} }
\thanks{Partially supported by University of Padova under the grant ``Progetto di Ateneo".}
}

\markboth{DRAFT}{Shell \MakeLowercase{\textit{et al.}}: Bare Demo of IEEEtran.cls for Journals}

\maketitle

\vspace{-1cm}

\begin{abstract}
In this paper, we consider an arbitrary matrix-valued, rational spectral density $\Phi(z)$. 
We show with a constructive proof that $\Phi(z)$ admits a factorization of the form 
$\Phi(z)=W^\top (z^{-1})W(z)$, where $W(z)$ is {\em stochastically minimal}. Moreover, $W(z)$ and its right inverse are analytic in regions that may be selected with the only constraint that they satisfy some symplectic-type conditions.
By suitably selecting the analyticity regions, this extremely general result particularizes into a corollary that may be viewed as the discrete-time counterpart of the matrix factorization method devised by Youla in his celebrated work \cite{Youla-1961}.
\end{abstract}

\IEEEpeerreviewmaketitle

\section{Introduction}

The spectral factorization problem is a classical and extensively investigated problem in Linear-Quadratic optimal control theory \cite{Willems-1971,Stoorvogel-S-98,Aksikas-et-al-07,Ferrante-Ntog-Automatica-13,Swigart-Lall-14}, estimation theory and stochastic realization
\cite{Lindquist-P-85-siam,Lindquist-P-91-jmsec,Picci-P-94,Ruckebusch-78-2,Ruckebusch-80,Ferrante-94-ieee,Ferrante-94-jmsec,Ferrante-M-P-93}, operator theory and network theory  \cite{Brune-31,Anderson-Vongpanitlerd-1973,BarasDW,FaurreCG,Fuhrmann-95,Helton,Zemanian}, interpolation theory |  from the classical paper
\cite{NudeScw} to the recent works of Byrnes, Georgiou, Lindquist and coworkers, see \cite{Byrnes-et-al} and references therein |  and passivity  from the classical positive-real systems theory \cite{Willems-1971,Anderson-Vongpanitlerd-1973,Brogliato-LME-07,Khalil-02} to the more recent negative-imaginary systems theory \cite{Petersen-Lanzon-10,Xiong-PL-10,Ferrante-N-13}, to mention just the main fields and a few references.
Indeed, spectral factorization  is the common denominator of a circle of ideas including LQ optimization methods, passivity theory, positivity, second-order stationary stochastic processes and Riccati equations. It seems therefore fair to say that spectral factorization is one of the cornerstones of modern systems and control theory.

Since the pioneering works of Kolmogorov and Wiener in the forties, a variety of methods have been proposed for the analysis and solution of this problem under different assumptions and in different settings, see e.g., \cite{Anderson-et-al-1974, Jezek-Kucera-1985, Rissanen-1973, Tunnicliffe-Wilson-1972,Callier-1985,Youla-Kazanjian-1978,Moir1,Moir2,Moir3}, to cite but a few.
We also refer to the relatively recent survey \cite{Sayed-K} that contains many other references and different points of view on this problem. A particularly relevant result on this topic is the well-known procedure devised by Youla in \cite{Youla-1961} which can be used to solve the rational multivariate spectral factorization problem in continuous-time. 
Remarkably, this method does not require any additional system-theoretic assumption: the rational spectrum $\Phi(s)$ may feature poles and zeroes on the imaginary axis, its rank may be deficient and it can be a non-proper rational function.
Moreover, this method permits a generalization that allows for the selection of the region of analyticity of the spectral factor. 
This turns out to be a crucial feature in the solution of related control problems: For example, in \cite{Ferrante-Pandolfi-2002} a spectral factor having poles and zeroes in a certain  region of the complex plane has been used to weaken the standard assumptions for the solvability of the classical Positive Real Lemma equations.

Surprisingly, the discrete-time counterpart of this result is so far missing. The reason could be due to the difficulty of deriving a result that parallels the Oono-Yasuura algorithm that constitutes a fundamental step in Youla's work. 
In order to fill this gap, in this paper, we establish a general discrete-time spectral factorization result.   In particular, we show that, given an arbitrary rational matrix function $\Phi(z)$ that is positive semi-definite on the unit circle, and two arbitrary regions featuring a geometry compatible with spectral factorization, $\Phi(z)$ admits a spectral factorization of the form  $\Phi(z)=W^\top (z^{-1})W(z)$ where the poles and zeroes of $W(z)$ lie  on the prescribed regions. The proof is constructive and gives, as a byproduct, stochastic minimality of  the spectral factor, (i.e. minimality of the McMillan degree of $W(z)$) which is a crucial feature in  stochastic realization theory \cite{Lindquist-P-85-siam,Lindquist-P-91-jmsec,Ferrante-96-siam} and is one of the key aspects in the present analysis.
We consider the factorization of the form $\Phi(z)=W^\top (z^{-1})W(z)$ corresponding to optimal control and network synthesis problems.  All the theory is, however, easily adaptable to obtain a dual counterpart for the factorization of the form $\Phi(z)=W(z)W^\top (z^{-1})$. The latter is  the natural factorization associated to  the representation of second-order stationary stochastic processes and hence to filtering and estimation problems. 
In fact, if $\Phi(z)$ is the spectral density of such a process $y(t)$, and $\Phi(z)$ admits a spectral factorization of the form $\Phi(z)=W(z)W^\top (z^{-1})$, then $y(t)$ may be represented as the output of a linear system with transfer function $W(z)$ driven by white noise $e(t)$. When all the poles of $W(z)$ lie inside the unit circle, $W(z)$ is called {\em causal} spectral factor as there is a causal relation between $e(t)$ and  $y(t)$ \cite{Lindquist-P-85-siam}. If, moreover, also the zeroes of $W(z)$ lie inside the unit circle, $W(z)$ is called {\em outer} spectral factor and the relation between  $y(t)$ and $e(t)$ (which is, in this case, the innovation of $y(t)$) is causal and causally invertible.   
The  outer spectral factor is essentially unique and may be recovered in our theory by suitably selecting the regions where the poles and zeroes of $W(z)$ are located; this may be viewed as the discrete-time counterpart of Youla's result.
Of course, with respect to most classical control applications, the outer spectral factor is the required solution. Nevertheless, when a-causal control and estimation problems are involved, see e.g.
\cite{Willems-1971,Colaneri-F-siam,Colaneri-F-SCL,Ferrante-P-98}, and in stochastic realization theory, see \cite{Lindquist-P-85-siam,Picci-P-94,Ferrante-P-P-02-LAA}, spectral factors whose poles and zeroes lie in different regions of the complex plane become important. This provides a strong motivation for our general result where the regions for poles and zeroes of the spectral factor can be suitably selected.

The organization of the paper is as follows.   In section \ref{sec:prob-def+main-res}, we formally introduce the discrete-time spectral factorization problem and, after a few definitions we present our main results. 
In section \ref{sec:problem-statement}, we review some notions from polynomial and rational matrix theory. Section \ref{sec:pre-analysis} is devoted to present a number of preliminary  results. In section \ref{sec:main-theorem}, we derive the proof of our main result and present some byproducts of our theory. Section VI shows a numerical example of the proposed factorization algorithm. Finally, in section \ref{sec:conclusions}, we draw some concluding remarks and we describe a number of possible future research directions.

{\em General notation and conventions:} Given an arbitrary matrix $G$, we write $G^\top$, $\overline{G}$, $G^{-1}$, $G^{-L}$ and $G^{-R}$ for the transpose, complex conjugate, inverse, left inverse and right inverse of $G$, respectively. In what follows, $[G]_{ij}$ stands for the $(i,j)$-th entry of $G$ and $[G]_{i:j,k:h}$ for the sub-matrix obtained by extracting the rows from index $i$ to index $j$ ($i\leq j$)  of $G$ and the columns from index $k$ to index $h$ ($k\leq h$) of $G$. If $\mathbf{v}$ is a vector, then $[\mathbf{v}]_i$ denotes the $i$-th component of $v$. Here, as usual, $I_n$ is the $n\times n$ identity matrix, $\mathbf{0}_{m,n}$ is the $m\times n$ zero matrix and $\diag[a_1,\dots,a_n]$ represents the matrix whose diagonal entries are $a_1,\dots,a_n$.

We denote by $\R[z]^{m\times n}$, $\R[z,z^{-1}]^{m\times n}$ and $\R(z)^{m\times n}$ the set of real $m\times n$ polynomial, Laurent polynomial (L-polynomial, for short) and rational matrices, respectively. Given a rational matrix $G(z)\in\R(z)^{m\times n}$, we let $G^*(z):=G^\top(z^{-1})$, $G^{-*}(z):=[G^{-1}]^*(z)$, $G^{-R*}(z):=[G^{-R}]^*(z)$ and $G^{-L*}(z):=[G^{-L}]^*(z)$. We denote by $\rk(G)$ the normal rank of $G(z)$, i.e., the rank almost everywhere in $z\in\C$ of $G(z)$. The rational matrix $G(z)$ is said to be analytic in a region of the complex plane if all its entries are analytic in this region. Moreover, as in \cite{Youla-1961}, with a slight abuse of notation, when we say that a rational function $f(z)$ is analytic in   a region $\mathbb{T}$ of the complex plane that is not open, we mean that $f(z)$ does not have poles in $\mathbb{T}$. In the case of a rational $f(z)$ this abuse does not cause any problems; in fact, $f(z)$ can have only finitely many poles so that there exists a larger open region $\mathbb{T}_\varepsilon\supset \mathbb{T}$ in which $f(z)$ is indeed analytic. For example, if $f(z)$ is rational and does not have poles on the unit circle, we say that $f(z)$ is analytic on the unit circle in place of $f(z)$ is analytic on an open annulus containing the unit circle. Notice that  such an annulus does indeed exist.

Finally, throughout the paper, we let $\R_0:=\R\setminus\{0\}$, $\C_0:=\C\setminus\{0\}$ and we denote by $\overline{\C}:=\C\cup \{\infty\}$ the extended complex plane.

\section{Problem definition and Main result}\label{sec:prob-def+main-res}

We start by introducing the object of our analysis and define the problem of spectral factorization:
\begin{definition}[Para-Hermitian matrix]
A rational matrix $G(z)\in\mathbb{R}(z)^{n\times n}$ is said to be \emph{para-Hermitian} if $G(z) = G^*(z)$.
\end{definition}

\begin{definition}[Spectrum]
A para-Hermitian rational matrix $\Phi(z)\in\mathbb{R}(z)^{n\times n}$ is said to be a \emph{spectrum} if $\Phi(e^{j\omega})$ is positive semi-definite for all $\omega\in[0,2\pi)$ such that $\Phi(e^{j\omega})$ is defined.
\end{definition}

\begin{definition}[Para-unitary matrix]
A rational matrix $G(z)\in\mathbb{R}(z)^{n\times n}$ is said to be \emph{para-unitary} if 
\[
G^*(z) G(z) =G(z) G^*(z)=I_n.
\]
\end{definition}
\begin{remark}
Notice that a para-Hermitian matrix $G(z)$ is Hermitian in the ordinary sense on the unit circle, while a para-unitary matrix $G(z)$ is unitary in the ordinary sense on the unit circle.
\end{remark}

The spectral factorization problem can be defined as follows:
\begin{problem}\label{prob-sf}
Given a spectrum $\Phi(z)$ find a factorization of the form
\begin{equation}\label{sp-fac-def}
\Phi(z)=W^*(z)W(z).
\end{equation}
\end{problem}

A matrix function $W(z)$ satisfying (\ref{sp-fac-def}) is called {\em spectral factor} of $\Phi(z)$.
Clearly, Problem \ref{prob-sf} admits many solutions. For control applications we are interested in solutions featuring some additional properties: Typical requirements are minimal complexity | as measured by the McMillan degree of $W(z)$ | full  row-rank of $W(z)$, and the fact that the poles and/or the zeroes of $W(z)$ lie in certain regions of the complex plane.
The most general kind of such regions are the following.

\begin{definition}[(Weakly) Unmixed-symplectic]\label{def:unmixed-symplectic}
A set $\mathscr{A}\subset \overline{\C}$ is {\em unmixed-symplectic}\footnote{The reason for the term ``symplectic'' is that   $\mathscr{A}$ and $\mathscr{A}^*$ are symmetric with respect to the unit circle, a type of symmetry induced by symplectic property, see, e.g. \cite{Ferrante-L-98}.  In this spirit, the corresponding property in continuous-time, where $\mathscr{A}^*:= \{\,z\,:\, -z\in\mathscr{A}\,\}$, $\mathscr{A}\cup \mathscr{A}^*$ is the whole complex plane with the exception of the imaginary axis and $\mathscr{A}\cap \mathscr{A}^*=\emptyset$, could be called ``unmixed-Hamiltonian".
} if 
$$
\mathscr{A}\cup \mathscr{A}^*=\overline{\C}\setminus \{\,z\in\C\,:\, |z|=1\,\},\ \ {\rm and}\ \ 
\mathscr{A}\cap \mathscr{A}^*=\emptyset,$$
where $\mathscr{A}^*=\{\,z\,:\, z^{-1}\in\mathscr{A}\,\}$.
The set $\mathscr{A}\subset \overline{\C}$ is {\em weakly unmixed-symplectic} if 
$$
\mathscr{A}\cup \mathscr{A}^*=\overline{\C},\ \ {\rm and}\ \ 
\mathscr{A}\cap \mathscr{A}^*=\{\,z\in\C\,:\, |z|=1\,\}.$$

\end{definition}   

We are now ready for our main result.

\begin{theorem}\label{thmsf-dt-g}
Let $\Phi(z)\in\R(z)^{n\times n}$ be a spectrum of normal rank $\mathrm{rk}(\Phi)=r\neq 0$.
Let $\mathscr{A}_p$ and $\mathscr{A}_z$ be two unmixed-symplectic sets.
Then, there exists a function $W(z)\in\R(z)^{r\times n}$ such that
\begin{enumerate}
\item $\Phi(z)=W^*(z)W(z)$. \label{item:thmsf-dt-g(1)}
\item $W(z)$ is analytic in $\mathscr{A}_p$ and its right inverse $W^{-R}(z)$ is analytic in $\mathscr{A}_z$. \label{item:thmsf-dt-g(2)}
\item \label{item:thmsf-dt-g(3)} $W(z)$ is {\em stochastically minimal}, i.e., the McMillan degree of $W(z)$ is a half of the McMillan degree of $\Phi(z)$.
\newcounter{temp}
\setcounter{temp}{\value{enumi}}
\end{enumerate}
Moreover,
\begin{enumerate}
\setcounter{enumi}{\value{temp}}
\item 
\label{item:thmsf-dt-g(4)}
If $\mathscr{A}_p=\mathscr{A}_z$ then 
$W(z)$ satisfying points \ref{item:thmsf-dt-g(1)}), and \ref{item:thmsf-dt-g(2)}) is unique up to a constant, orthogonal matrix multiplier on the left, i.e., if $W_1(z)$ also satisfies points \ref{item:thmsf-dt-g(1)}), and \ref{item:thmsf-dt-g(2)})  then  $W_1(z)=TW(z)$ where $T\in\R^{r\times r}$ is orthogonal.
Therefore, if $\mathscr{A}_p=\mathscr{A}_z$, points \ref{item:thmsf-dt-g(1)}) and \ref{item:thmsf-dt-g(2)}) imply point \ref{item:thmsf-dt-g(3)}).
\item
\label{item:thmsf-dt-g(6)} 
If  $\Phi(z)=L^*(z) L(z)$ is any factorization  in which $L(z)\in\R(z)^{r\times n}$ is analytic in $\mathscr{A}_z$, then $L(z)=V(z)W(z)$, $V(z)\in\R(z)^{r\times r}$ being a para-unitary matrix analytic in $\mathscr{A}_z$. 
Moreover, given an arbitrary para-unitary matrix $V(z)\in\R(z)^{r\times r}$ being analytic in $\mathscr{A}_p$,  $L(z):=V(z)W(z)$ is analytic in $\mathscr{A}_p$ and satisfies $\Phi(z)=L^*(z) L(z)$, so that, if $\mathscr{A}_p=\mathscr{A}_z=:\mathscr{A}$ then 
$\Phi(z)=L^*(z) L(z)$ is a factorization  in which $L(z)\in\R(z)^{r\times n}$ is analytic in $\mathscr{A}$ if and only if $L(z)=V(z)W(z)$, $V(z)\in\R(z)^{r\times r}$ being a para-unitary matrix analytic in $\mathscr{A}$.
\item If $\Phi(z)$ is analytic on the unit circle, then points \ref{item:thmsf-dt-g(1)})-\ref{item:thmsf-dt-g(6)}) still hold even if $\mathscr{A}_p$ is weakly unmixed-symplectic.
\label{item:thmsf-dt-g(7)}
\item If $\Phi(z)$ is analytic on the unit circle and the rank of $\Phi(z)$ is constant on the unit circle, then points \ref{item:thmsf-dt-g(1)})-\ref{item:thmsf-dt-g(6)}) still hold even if $\mathscr{A}_p$ and/or $\mathscr{A}_z$ are weakly unmixed-symplectic.
 \label{item:thmsf-dt-g(8)}
\end{enumerate}
\end{theorem}

Of course the most common requirement in control theory is that $W(z)$ is outer which correspond to setting $\mathscr{A}_p=\mathscr{A}_z=\{\,z\in\overline{\C}\,:\,|z|>1 \,\}$ in the general case, $\mathscr{A}_p=\{\,z\in\overline{\C}\,:\,|z|\geq 1 \,\}$ and $\mathscr{A}_z=\{\,z\in\overline{\C}\,:\,|z|>1 \,\}$
in the case when $\Phi(z)$ is analytic on the unit circle and $\mathscr{A}_p=\mathscr{A}_z=\{\,z\in\overline{\C}\,:\,|z|\geq1 \,\}$ when $\Phi(z)$ is analytic on the unit circle and the rank of $\Phi(z)$ is constant on the unit circle. This particular case of the previous result corresponds to the following result whose first 6 points are the discrete-time counterpart of the celebrated Youla's Theorem \cite[Thm.2]{Youla-1961}.

\begin{theorem}\label{thmsf-dt}
Let $\Phi(z)\in\R(z)^{n\times n}$ be a   spectrum of normal rank  $\mathrm{rk}(\Phi)=r\neq 0$. Then, there exists a matrix $W(z)\in\R(z)^{r\times n}$ such that
\begin{enumerate}
\item $\Phi(z)=W^*(z)W(z)$. \label{item:thmsf-dt(i)}
\item $W(z)$ and its right inverse $W^{-R}(z)$ are both analytic in $\{\,z\in\overline{\C}\,:\,|z|>1 \,\}$. \label{item:thmsf-dt(ii)}
\item $W(z)$ is unique up to a constant, orthogonal matrix multiplier on the left, i.e., if $W_1(z)$ also satisfies points \ref{item:thmsf-dt(i)}) and \ref{item:thmsf-dt(ii)}), then  $W_1(z)=TW(z)$ where $T\in\R^{r\times r}$ is orthogonal. \label{item:thmsf-dt(iii)}
\item Any factorization of the form $\Phi(z)=L^*(z) L(z)$ in which $L(z)\in\R(z)^{r\times n}$ is analytic in $\{\,z\in\overline{\C}\,:\,|z|>1 \,\}$, has the form $L(z)=V(z)W(z)$, where $V(z)\in\R(z)^{r\times r}$ is a para-unitary matrix analytic in $\{\,z\in\overline{\C}\,:\,|z|>1 \,\}$. Conversely, any $L(z)=V(z)W(z)$, where $V(z)\in\R(z)^{r\times r}$ is a para-unitary matrix analytic in $\{\,z\in\overline{\C}\,:\,|z|>1 \,\}$, is a spectral factor of $W(z)$ analytic in $\{\,z\in\overline{\C}\,:\,|z|>1 \,\}$. \label{item:thmsf-dt(iv)}
\item If $\Phi(z)$ is analytic on the unit circle, then $W(z)$ is analytic in $\{\, z\in\overline{\C}\,:\,|z|\geq 1\, \}$. \label{item:thmsf-dt(v)}
\item If $\Phi(z)$ is analytic on the unit circle and the rank of $\Phi(z)$ is constant on the unit circle, then $W(z)$ and its right inverse $W^{-R}(z)$ are both analytic in $\{\, z\in\overline{\C}\,:\,|z|\geq 1\, \}$.
 \label{item:thmsf-dt(vi)}
\item \label{item:thmsf-dt(vii)} $W(z)$ satisfying points \ref{item:thmsf-dt(i)}) and
\ref{item:thmsf-dt(ii)}) is {\em stochastically minimal}, i.e., the McMillan degree of $W(z)$ is a half of the McMillan degree of $\Phi(z)$.
\end{enumerate}
\end{theorem}

\begin{remark}
Notice that the hypothesis   $\mathrm{rk}(\Phi)\neq 0$ of the previous results is only assumed to rule out the trivial case of an identically zero spectrum  $\Phi(z)$ for which the only spectral factorizations   clearly correspond to $W(z)=\mathbf{0}_{m,n}$, with $m$ being arbitrary, so that, in this case, $W(z)$ cannot be chosen to be full  row-rank.
\end{remark}

\section{Mathematical preliminaries on rational matrices}\label{sec:problem-statement}

Let $f(z)=p(z)/q(z)\in \R(z)$, $q(z)\neq 0$, be a nonzero rational function. We can always write $f(z)$ in the form
\[
f(z)=\frac{n(z)}{d(z)}(z-\alpha)^\nu, \quad \forall\, \alpha\in \C,
\]
where $\nu$ is an integer and $n(z),\, d(z)\in \R[z]$ are nonzero polynomials such that $n(\alpha)\neq 0$ and $d(\alpha)\neq 0$. The integer $\nu$ is called {\em valuation of f(z) at $\alpha$} and we denote it with the symbol $v_\alpha(f)$. The valuation of $f(z)$ at infinity is defined as $v_\infty(f):=\deg q(z)-\deg p(z)$, where $\deg(\cdot)$ denotes the degree of a polynomial.  If $f(z)$ is the null function, by convention,  $v_\alpha(f)=+\infty$ for every $\alpha\in \overline{\C}$. If $v_\alpha(f)<0$, then $\alpha\in \overline{\C}$ is called a {\em pole} of $f(z)$ of multiplicity $-v_\alpha(f)$. If $v_\alpha(f)>0$, then $\alpha\in \overline{\C}$ is called a {\em zero} of $f(z)$ of multiplicity $v_\alpha(f)$. The rational function $f(z)$ is said to be \emph{proper} if $v_\infty(f)\geq 0$, \emph{strictly proper} if $v_\infty(f)> 0$.

A polynomial matrix $G(z)\in\mathbb{R}[z]^{m\times n}$ is said to be {\em unimodular} if it has a polynomial inverse (either left, right or both). Similarly, a L-polynomial matrix $G(z)\in\mathbb{R}[z,z^{-1}]^{m\times n}$ is said to be {\em L-unimodular} if it has a L-polynomial inverse (either left, right or both). A square polynomial matrix $G(z)\in\mathbb{R}[z]^{n\times n}$ is unimodular if and only if its determinant is a nonzero constant $\alpha\in\R_0$. On the other hand, a square L-polynomial matrix $G(z)\in\mathbb{R}[z,z^{-1}]^{n\times n}$ is L-unimodular if and only if its determinant is a nonzero monomial $\alpha z^{k}$, $\alpha\in\R_0$, $k\in\mathbb{Z}$.

Consider now a nonzero real L-polynomial vector $\mathbf{v}(z)\in\R[z,z^{-1}]^{p}$. We can write it as
\[
\mathbf{v}(z)=\mathbf{v}_k z^k+\mathbf{v}_{k+1} z^{k+1}+\cdots+\mathbf{v}_{K-1} z^{K-1}+\mathbf{v}_K z^K,
\]
with $\mathbf{v}_k$ and $\mathbf{v}_K$, $k\leq K$, nonzero vectors in $\R^p$. We say that the integer $k$ is the {\em minimum-degree} of $\mathbf{v}(z)$, written $\mindeg \mathbf{v}$, while the integer $K$ is the {\em maximum-degree} of $\mathbf{v}(z)$, written $\maxdeg \mathbf{v}$.\footnote{If $\mathbf{v}(z)$ is the zero vector, then $\mindeg \mathbf{v}$ and $\maxdeg \mathbf{v}$ are left undefined.}

Let $G(z)\in\R[z,z^{-1}]^{m\times n}$ and let $k_i$ and $K_i$ be the minimum- and maximum-degree of the $i$-th column of $G(z)$, for all $i=1\dots,m$. We define the {\em highest-column-degree coefficient matrix} of $G(z)$ as the constant matrix $G^{\rm hc}\in\R^{m\times n}$ whose $i$-th column consists of the coefficients of the monomials $z^{K_i}$ in the same column of $G(z)$. Furthermore, we define the {\em lowest-column-degree coefficient matrix} of $G(z)$ as the constant matrix $G^{\rm lc}\in\R^{m\times n}$ whose $i$-th column consists of the coefficients of the monomials $z^{k_i}$ in the same column of $G(z)$. By considering, instead of the columns, the rows of $G(z)$ we can define, by following the same lines in the above, the {\em highest-row-degree coefficient matrix} of $G(z)$, $G^{\rm hr}\in\R^{m\times n}$, and the {\em lowest-row-degree coefficient matrix} of $G(z)$, $G^{\rm lr}\in\R^{m\times n}$.

A classical result in rational matrix theory is the following (see, e.g., \cite[Ch.6, \S5]{Kailath-1998}).
\begin{theorem}[Smith-McMillan]
\label{thm:smith-mcmillan}
Let $G(z)\in\R(z)^{m\times n}$ and let $\rk(G)=r$. There exist unimodular matrices $U(z)\in\mathbb{R}[z]^{m\times r}$ and $V(z)\in\mathbb{R}[z]^{r\times n}$ such that
\begin{align}\label{eq:smith-mcmillan-canonic-form}
D(z):&=U(z)G(z)V(z)\nonumber\\
&=\diag\left[\frac{\varepsilon_1(z)}{\psi_1(z)},\frac{\varepsilon_2(z)}{\psi_2(z)},\dots,\frac{\varepsilon_r(z)}{\psi_r(z)}\right],
\end{align}
where $\varepsilon_1(z),\varepsilon_2(z),\dots,\varepsilon_r(z), \psi_1(z), \psi_2(z),\dots,\psi_r(z)\in\R[z]$ are monic polynomials satisfying the conditions: {(i)} $\varepsilon_i(z)$ and $\psi_i(z)$ are relatively prime, $i=1,2,\dots,r$, {(ii)} $\varepsilon_i(z)\mid \varepsilon_{i+1}(z)$ and $\psi_{i+1}(z)\mid \psi_{i}(z)$, $i=1,2,\dots,r-1$.{\footnote{We write $p(z) \mid q(z)$, with $p(z), q(z)\in \R[z]$, to say that $p(z)$ divides $q(z)$.}}
\end{theorem}

The rational matrix $D(z)$ in (\ref{eq:smith-mcmillan-canonic-form}) is known as the {\em Smith-McMillan canonical form} of $G(z)$. (In general, we say that a rational matrix is {\em canonic} if it is of the form in (\ref{eq:smith-mcmillan-canonic-form}) and satisfies the conditions of the above theorem.) The (finite) zeroes of $G(z)$ coincide with the zeroes of $\varepsilon_r(z)$ and the (finite) poles of $G(z)$ with the zeroes of $\psi_1(z)$. Note that, unlike what happens in the scalar case, the set of zeroes and poles of a rational matrix may not be disjoint.

Let $G(z)\in\mathbb{R}(z)^{m\times n}$ and write $G(z)=C(z)D(z)F(z)$, where $D(z)$ is the Smith-McMillan form of $G(z)$ and $C(z), F(z)$ are unimodular matrices. If $\mathrm{rk}(G)=m=n$, then the inverse of $G(z)$ has the form
\[
G^{-1}(z)=F^{-1}(z)D^{-1}(z)C^{-1}(z)
\]
and $D^{-1}(z)$ coincides with the Smith-McMillan canonical form of $G^{-1}(z)$, up to a permutation of the diagonal elements. Therefore, the poles of $G^{-1}(z)$ are exactly the zeroes of $G(z)$. In a similar fashion, if $G(z)$ has normal rank $m$ ($n$), there always exists a right (left) inverse of $G(z)$ such that the poles of $G^{-R}(z)$ ($G^{-L}(z)$) coincide with the zeroes of $G(z)$.\footnote{ The latter fact is not true for all the right/left inverses of $G(z)$, since, in general, the zeroes of $G(z)$ are among the poles of all such inverses (see \cite[Ch.6, \S5, Ex.14]{Kailath-1998}).} Indeed, we may take
\begin{align}
G^{-R}(z)&=F^{-R}(z)D^{-1}(z)C^{-1}(z),\label{eq:right-inv}\\
G^{-L}(z)&=F^{-1}(z)D^{-1}(z)C^{-L}(z).\label{eq:left-inv}
\end{align} 
In the following, we consider only right and left inverses of the form (\ref{eq:right-inv}) and (\ref{eq:left-inv}), respectively.

Let $\alpha_1,\alpha_2,\dots,\alpha_t$ be the (finite) zeroes and (finite) poles of $G(z)$.  We can write the Smith-McMillan canonical form of $G(z)$ as 
\begin{align}
\mathrm{diag}\Big[(z-\alpha_1)^{\nu_1^{(1)}}\cdots (z-&\alpha_t)^{\nu_t^{(1)}},\dots,\nonumber\\ &(z-\alpha_1)^{\nu_1^{(r)}}\cdots (z-\alpha_t)^{\nu_t^{(r)}}\Big].\nonumber
\end{align}
The integer exponents $\nu_i^{(1)}\leq \nu_i^{(2)}\leq \cdots \leq \nu_i^{(r)}$, appearing in the above expression, are called the {\em structural indices} of $G(z)$ at $\alpha_i$ and they are used to represent the zero-pole structure at $\alpha_i$ of $G(z)$. To obtain the zero-pole structure at infinity of $G(z)$, we can proceed as follows. We make a change of variable, $z\to \lambda^{-1}$, and compute the Smith-McMillan form of $G(\lambda^{-1})$, then the structural indices of $G(\lambda^{-1})$ at $\lambda=0$ will give the set of structural indices of $G(z)$ at $z=\infty$.
 Lastly, if $p_1,\dots,p_h$ are the distinct poles (the pole at infinity included) of $G(z)$, we recall that the {\em McMillan degree} of $G(z)$ can be defined as (see, e.g., \cite[Ch.6, \S5]{Kailath-1998})
\begin{equation}\label{eq:mcmillan-degree}
\delta_M(G):=\sum_{i=1}^h\delta(G;p_i),
\end{equation}
where $\delta(G;p_i)$ is the degree of the pole $p_i$, i.e., the largest multiplicity that $p_i$ possesses as a pole of any minor of $G(z)$. In particular, if $D(z)$ in (\ref{eq:smith-mcmillan-canonic-form}) is the Smith-McMillan form of $G(z)$ and $G(z)$ has no pole at infinity then $\delta(G;p_i)=\delta(D;p_i)$ for all $i=1,\dots,h$, which, in turn, yields $\delta_M(G)=\delta_M(D)=\sum_{i=1}^r \deg \psi_i(z)$.


\section{Preliminary results}\label{sec:pre-analysis}
In this section, we collect a set of  lemmata which we will exploit in the constructive proof of the main theorem.

\begin{lemma}\label{lemma1}
A matrix $G(z)\in\R(z)^{m\times n}$ is analytic in $\mathbb{C}_0$ together with its inverse (either right, left or both) if and only if it is a L-unimodular polynomial matrix.
\end{lemma}

\begin{proof}
If $G(z)$ is L-unimodular, then $G(z)$ has an inverse (either left, right or both) which is L-polynomial. Hence, the only possible finite zeroes/poles of $G(z)$ are located at $z=0$. This, in turn, implies that $G(z)$ must be analytic together with its inverse in $\C_0$.

Vice versa, suppose that $G(z)$ is analytic with its inverse in $\C_0$. Firstly, we notice that the existence of a left or right inverse for $G(z)$ implies that the normal rank of $G(z)$ is either $r=n$ or $r=m$, respectively. Without loss of generality, we can suppose that $r=n$. By the Smith-McMillan Theorem, we can write $G(z)=C(z)D(z)F(z)$, where $C(z)\in\mathbb{R}[z]^{m\times n},\  F(z)\in\mathbb{R}[z]^{n\times n}$ are unimodular (and, a fortiori, L-unimodular) polynomial matrices, respectively, and $D(z)\in\mathbb{R}(z)^{n\times n}$ is diagonal, canonic of the form 
\[
D(z)=\mathrm{diag} \left[\frac{\varepsilon_1(z)}{\psi_1(z)}, \frac{\varepsilon_2(z)}{\psi_2(z)},\dots , \frac{\varepsilon_n(z)}{\psi_n(z)} \right].
\]
The analyticity of $G(z)$ in $\C_0$ implies that all $\psi_i(z)\in\R[z]$, $i=1,\dots,n$, are nonzero monomials. The Smith-McMillan canonical form of $G^{-L}(z)$ is given by 
\[
\mathrm{diag} \left[\frac{\psi_n(z)}{\varepsilon_n(z)}, \frac{\psi_{n-1}(z)}{\varepsilon_{n-1}(z)},\dots , \frac{\psi_1(z)}{\varepsilon_1(z)} \right].
\]
Hence, the analyticity of $G^{-L}(z)$ in $\C_0$ implies that all $\varepsilon_i(z)\in\R[z]$, $i=1,\dots,n$, are nonzero monomials. Therefore, $D(z)$ is a L-unimodular polynomial matrix. Since $G(z)=C(z)D(z)F(z)$ is the product of three L-unimodular polynomial matrices, $G(z)$ must be a L-unimodular polynomial matrix.
\end{proof}

\begin{lemma}\label{lemma2}
Let $\mathscr{A}\subset \overline{\C}$ be an unmixed-symplectic set. A para-unitary matrix $G(z)\in\R(z)^{n\times n}$ analytic in $\mathscr{A}$ with inverse analytic in $\mathscr{A}$ is a constant orthogonal matrix.
\end{lemma}
\begin{proof}
The analyticity of the inverse of  $G(z)$ in $\mathscr{A}$ implies that of $G(z^{-1})$ in the same region, and therefore that of $G(z)$ in $\mathscr{A}^*$.
We also notice that in the unit circle it holds $G^*(e^{j\omega}) G(e^{j\omega})=G^\top(e^{-j\omega}) G(e^{j\omega})=I_n$, $\forall\,\omega \in [0,2\pi)$, and we can write out the diagonal elements in expanded form as
\[
\sum_{i=1}^n |[G(e^{j\omega})]_{ik}|^2=1, \ \ \ \ \forall\,  k=1,\dots,n, \ \forall\, \omega \in [0,2\pi).
\]
The latter equation implies that
\[
|[G(e^{j\omega})]_{ik}|\leq 1, \ \ \ \ \forall\, i,\, k=1,\dots,n, \ \forall\, \omega \in [0,2\pi),
\]
and, therefore, we proved the analyticity of $G(z)$ on the unit circle. By Definition \ref{def:unmixed-symplectic} of unmixed-symplectic set, it follows that $G(z)$ is analytic on the entire extended complex plane. This means that $G(z)$ is analytic and bounded in $\C$. Hence, we can apply Liouville's Theorem \cite[Ch.V,  Thm.1.4]{Lang-1985} and conclude that $G(z)$ must be a constant orthogonal matrix.
\end{proof}
\begin{remark}
With the usual choice $\mathscr{A}=\{\, z\in\overline{\C} \,:\, |z|> 1\,\}$, the previous lemma reads as follows: A para-unitary matrix $G(z)\in\R(z)^{n\times n}$ analytic in $\{\, z\in\overline{\C} \,:\, |z|> 1\,\}$ with inverse analytic in $\{\, z\in\overline{\C} \,:\, |z|> 1\,\}$ is a constant orthogonal matrix.
\end{remark}

\begin{definition}[Left-standard factorization]\label{def:ls-fact}
Let $G(z)\in \R(z)^{m\times n}$ and let $\mathrm{rk}(G)=r$. A decomposition of the form $G(z)=A(z)\Delta(z)B(z)$ 
is called a \emph{left-standard factorization} if
\begin{enumerate}
\item $\Delta(z)\in\R(z)^{r \times r}$ is diagonal and analytic with its inverse in $\{\, z\in \mathbb{C}_0 \,:\, |z|\neq1 \, \}$;
\item $A(z)\in\R(z)^{m \times r}$ is analytic together with its left inverse in $\{\, z\in \mathbb{C}_0 \,:\, |z|\leq 1 \, \}$;
\item $B(z)\in\mathbb{R}(z)^{r \times n}$ is analytic together with its right inverse in $\{\, z\in \mathbb{C} \,:\, |z|\geq 1 \, \}$.
\end{enumerate}
\end{definition}
\begin{remark}
If, in Definition \ref{def:ls-fact}, $A(z)$ and $B(z)$ are interchanged, we have a \emph{right-standard factorization}. Hence, it follows that any left-standard factorization of $G(z)$ generates a right-standard factorization of $G^\top(z)$, $G^{-1}(z)$ (if $G(z)$ is nonsingular), $G(z^{-1})$, e.g., in the first case we have $G^\top(z) =B^\top(z) \Delta(z) A^\top(z)$.
\end{remark}

\begin{lemma}\label{lemma3}
Any rational matrix $G(z)\in \R(z)^{m\times n}$ of normal rank $\mathrm{rk}(G)=r$ admits a left-standard factorization. 
\end{lemma}
\begin{proof}
By the Smith-McMillan Theorem, we can write $G(z)=C(z)D(z)F(z)$, where $C(z)\in \mathbb{R}[z]^{m\times r}$, $F(z)\in \mathbb{R}[z]^{r\times n}$ are unimodular polynomial matrices and $D(z)\in \mathbb{R}(z)^{r\times r}$ is diagonal and canonic of the form 
\[
D(z)=\mathrm{diag} \left[\frac{\varepsilon_1(z)}{\psi_1(z)}, \frac{\varepsilon_2(z)}{\psi_2(z)},\dots , \frac{\varepsilon_r(z)}{\psi_r(z)} \right].
\]
We factor $\varepsilon_i(z)\in\R[z]$ and $\psi_i(z)\in\R[z]$, $i=1,\dots,r$, in $D(z)$ into the product of three polynomials: the first without zeroes in $\{\, z\in \C\,:\, |z|\leq 1 \,\}$, the second without zeroes in $\{\, z\in \C\,:\, |z|\neq 1 \,\}$ and the third without zeroes in $\{\, z\in \C\,:\, |z|\geq 1 \,\}$. Thus, it is possible to write
\[
D(z)=D_-(z)\Delta(z)D_+(z),
\] 
where $D_-(z)$ and its inverse are analytic in $\{\, z\in \C\,:\, |z|\leq 1 \,\}$, $\Delta(z)$ and its inverse in $\{\, z\in \C\,:\, |z|\neq 1 \,\}$ and $D_+(z)$ and its inverse in $\{\, z\in \C\,:\, |z|\geq 1 \,\}$. Eventually, by choosing $A(z):=C(z)D_-(z)$ and $B(z):=D_+(z)F(z)$, we have that $G(z)=A(z)\Delta(z)B(z)$ is a left-standard factorization of $G(z)$. 
\end{proof}

Left-standard factorizations are not unique. Indeed, any two decompositions are connected as follows.

\begin{lemma}\label{lemma4}
Let $G(z)\in \R(z)^{m\times n}$ be a rational matrix of normal rank $\mathrm{rk}(G)=r$ and let $G(z)=A(z)\Delta(z)B(z)=A_1(z)\Delta_1(z)B_1(z)$ be two left-standard factorizations of $G(z)$. Then, 
\[
A_1(z)=A(z)M^{-1}(z), \quad B_1(z)=N(z)B(z),
\]
where $M(z)\in\R[z,z^{-1}]^{r\times r}$ and $N(z)\in\R[z,z^{-1}]^{r\times r}$ are two L-unimodular polynomial matrices such that
\begin{align}\label{eq:M(z)Delta(z)N(z)}
M(z)\Delta(z)N^{-1}(z)=\Delta_1(z).
\end{align}
\end{lemma}
\begin{proof}
By assumption, $$G(z)=A(z)\Delta(z)B(z)=A_1(z)\Delta_1(z)B_1(z)$$  which, in turn, implies
\begin{align}\label{eq:lemma4}
\Delta_1^{-1}(z)A_1^{-L}(z)A(z)\Delta(z)=B_1(z)B^{-R}(z).
\end{align}
By Definition \ref{def:ls-fact} of left-standard factorization, the right-hand side of (\ref{eq:lemma4}) is analytic in $\{\,z\in\C\,:\,|z|\geq 1\,\}$, while the left-hand side of (\ref{eq:lemma4}) in $\{\,z\in\C_0\,:\,|z|< 1\,\}$. Therefore, it follows that $B_1(z)B^{-R}(z)$ is analytic in $\C_0$. Moreover, the inverse of $B_1(z)B^{-R}(z)$ satisfies
\[
[B_1(z)B^{-R}(z)]^{-1}=\Delta^{-1}(z)[A_1^{-L}(z)A(z)]^{-1}\Delta_1(z)
\] 
and is also analytic in $\C_0$. Thus, by Lemma \ref{lemma1}, $N(z):=B_1(z)B^{-R}(z)$ must be a L-unimodular matrix. Similarly, $M(z):=A_1^{-L}(z)A(z)$ is a L-unimodular matrix. Finally, a rearrangement of (\ref{eq:lemma4}) yields (\ref{eq:M(z)Delta(z)N(z)}).
\end{proof}
\begin{remark}
Notice that, by replacing the word ``left-standard'' with the word ``right-standard'' in Lemmata \ref{lemma3} and \ref{lemma4}, we obtain, by minor modifications in the proofs, a right-standard counterpart of Lemmata \ref{lemma3} and \ref{lemma4}.
\end{remark}

Let $\Phi(z)\in\R(z)^{n\times n}$ be a para-Hermitian matrix of normal rank $\mathrm{rk}(\Phi)=r$ and let $\Phi(z)=A(z)\Delta(z)B(z)$ be a left-standard factorization of $\Phi(z)$. We have that 
\[
\Phi(z)=\Phi^*(z) =B^*(z) \Delta^*(z) A^*(z)
\] 
is also a left-standard factorization of $\Phi(z)$. In particular, $\Delta^*(z)$ is equal to $\Delta(z)$, except for multiplication of suitable monomials of the form $\pm z^{k_i}$ in its diagonal entries, i.e.,
\[
\Delta^*(z) =\Sigma(z) \Delta(z),
\]
where 
\begin{equation}\label{eq:sigmadelta}
\Sigma(z)=\mathrm{diag}\left[e_1(z),e_2(z),\dots,e_r(z)\right]
\end{equation} 
and $e_i(z)=\pm z^{k_i},\ k_i\in\mathbb{Z},\ i=1,\dots,r$.
By invoking Lemma \ref{lemma4}, we can write
\begin{equation}\label{eq:A(z)B(z)}
A^*(z) = N(z)B(z),  \quad B^*(z) = A(z)M^{-1}(z),
\end{equation}
where $N(z),\, M(z)\in\R[z,z^{-1}]^{r\times r}$ are L-unimodular matrices.

The following two lemmata are used to establish a further characterization of a para-Hermitian matrix when it is positive semi-definite upon the unit circle.

\begin{lemma}\label{lemma5-pre}
Let $G(z)\in\R(z)^{n\times n}$ and let $\mathbb{T}$ be a region of the complex plane such that 
\begin{enumerate}
\item $G(z)$ is Hermitian on $\mathbb{T}$;
\item $\mathbf{x}^\top G(\lambda)\mathbf{x}\geq 0$, $\forall\,\mathbf{x}\in\R^n$ and $\forall\,\lambda\in\tilde{\mathbb{T}}\subseteq \mathbb{T}$ for which $G(\lambda)$ has finite entries.
\end{enumerate}
Let $D(z)\in\R(z)^{r\times r}$ be the Smith-McMillan canonical form of $G(z)$ and denote by $g_{\mathbf{ij}}^{(\ell)}$ and $d_{\mathbf{ij}}^{(\ell)}$ the $\ell\times \ell$ minors ($1\leq \ell\leq r$) of the rational matrices $G(z)$ and $D(z)$, respectively,  obtained by selecting those rows and columns whose indices appear in the ordered $\ell$-tuples $\mathbf{i}$ and $\mathbf{j}$, respectively. Then,
\[
\min_{\mathbf{i}} v_\alpha(d_{\mathbf{ii}}^{(\ell)})=\min_{\mathbf{i}} v_\alpha(g_{\mathbf{ii}}^{(\ell)}), \quad \forall \alpha\in\mathbb{T}.
\]
\end{lemma}

\begin{proof}
Firstly, we recall that for any rational matrix $G(z)$ it holds 
\[
\min_{\mathbf{i}} v_\alpha(d_{\mathbf{ii}}^{(\ell)})=\min_{\mathbf{ij}} v_\alpha(d_{\mathbf{ij}}^{(\ell)})=\min_{\mathbf{ij}} v_\alpha(g_{\mathbf{ij}}^{(\ell)}),\quad \forall \alpha\in\C. 
\] 
The latter result is well-known and is presented, for instance, as an exercise in \cite[Ch.6, \S 5, Ex.6]{Kailath-1998}. Hence, it remains to prove that 
\begin{equation}\label{eq:lemma-val-g}
\min_{\mathbf{ij}} v_\alpha(g_{\mathbf{ij}}^{(\ell)})=\min_{\mathbf{i}} v_\alpha(g_{\mathbf{ii}}^{(\ell)}),\quad \forall \alpha\in\mathbb{T}.
\end{equation}
Since $G(z)$ is Hermitian positive semi-definite on the region $\tilde{\mathbb{T}}$, it admits a decomposition of the form $G(\lambda)=W(\lambda)\overline{W(\lambda)}^\top$ for all  $\lambda\in\tilde{\mathbb{T}}$. By applying the Binet-Cauchy Theorem (see \cite[Vol.I, Ch.1, \S 2]{Gantmacher-1959}), we have
\begin{align}
&g_{\mathbf{ij}}^{(\ell)}(\lambda)=\sum_{\mathbf{h}}w_{\mathbf{ih}}^{(\ell)}(\lambda)\overline{w_{\mathbf{jh}}^{(\ell)}(\lambda)}, \quad \forall\,\lambda\in\tilde{\mathbb{T}}, \label{eq:corollary-minors-1}\\
&g_{\mathbf{ii}}^{(\ell)}(\lambda)=\sum_{\mathbf{h}}w_{\mathbf{ih}}^{(\ell)}(\lambda)\overline{w_{\mathbf{ih}}^{(\ell)}(\lambda)}=\sum_{\mathbf{h}}\left|w_{\mathbf{ih}}^{(\ell)}(\lambda)\right|^2, \quad \forall\,\lambda\in\tilde{\mathbb{T}},\label{eq:corollary-minors-2}
\end{align}
where $g_{\mathbf{ij}}^{(\ell)}(\lambda)$ and $w_{\mathbf{ij}}^{(\ell)}(\lambda)$ denote the $\ell\times \ell$ minors of matrices $G(\lambda)$ and $W(\lambda)$,  obtained by selecting those rows and columns whose indices appear in the ordered $\ell$-tuples $\mathbf{i}$ and $\mathbf{j}$, respectively. Moreover, in both the summations (\ref{eq:corollary-minors-1})-(\ref{eq:corollary-minors-2}), $\mathbf{h}:=(h_1,\dots,h_\ell)$, $1\leq h_1<\cdots<h_\ell\leq n$, runs through all such multi-indices. By using Cauchy-Schwarz inequality and (\ref{eq:corollary-minors-1})-(\ref{eq:corollary-minors-2}), we have
\begin{align}
\left|g_{\mathbf{ij}}^{(\ell)}(\lambda)\right|&=\left|\sum_{\mathbf{h}}w_{\mathbf{ih}}^{(\ell)}(\lambda)\overline{w_{\mathbf{jh}}^{(\ell)}(\lambda)}\right|\nonumber \\
&\leq \sqrt{\sum_{\mathbf{h}}\left|w_{\mathbf{ih}}^{(\ell)}(\lambda)\right|^2\sum_{\mathbf{h}}\left|w_{\mathbf{jh}}^{(\ell)}(\lambda)\right|^2}\nonumber \\
&= \sqrt{g_{\mathbf{ii}}^{(\ell)}(\lambda)g_{\mathbf{jj}}^{(\ell)}(\lambda)} \nonumber \\
&\leq  \max\left\{g_{\mathbf{ii}}^{(\ell)}(\lambda), g_{\mathbf{jj}}^{(\ell)}(\lambda) \right\}, \ \ \ \forall\,\lambda\in\tilde{\mathbb{T}}. \label{eq:inequality-cs}
\end{align}
The latter inequality implies that for every zero $\alpha\in\mathbb{T}$ of multiplicity $k$ of a minor of $G(z)$, there exists at least one principal minor of $G(z)$ which has the same $\alpha$ either as a zero  of multiplicity less than or equal to $k$ or a pole  of multiplicity greater than or equal to $0$. Similarly, inequality (\ref{eq:inequality-cs}) implies also that for every pole $\alpha\in\mathbb{T}$ of multiplicity $k$ of a minor of $G(z)$, there exists at least one principal minor of $G(z)$ which has the same pole of multiplicity greater than or equal to $k$. Therefore, we conclude that (\ref{eq:lemma-val-g}) holds.
\end{proof}

\begin{lemma}\label{lemma5}
Let $\Phi(z)\in\R(z)^{n\times n}$ be a spectrum of normal rank $\mathrm{rk}(\Phi)=r$ and let  $D(z)\in\R(z)^{r\times r}$ be its Smith-McMillan canonical form. Then, the zeroes and poles on the unit circle of the diagonal elements of $D(z)$ have even multiplicity.
\end{lemma}

\begin{proof}
Firstly, we assume that the numerators and denominators of all entries in $\Phi(z)$ are relatively prime polynomials. Let $
\alpha_1=e^{j\omega_1},\, \alpha_2=e^{j\omega_2},\dots,\, \alpha_t=e^{j\omega_t},
$ 
be the zeroes/poles on the unit circle of $\Phi(z)$  and let 
$
\nu_i^{(1)},\, \nu_i^{(2)},\dots,\, \nu_i^{(r)}, \ (\nu_i^{(1)}\leq \nu_i^{(2)}\leq \dots\leq\nu_i^{(r)}),
$
be the structural indices of $\Phi(z)$ at $\alpha_i, \ i=1,\dots,t$. Since $\Phi(z)$ is positive semi-definite on the unit circle, one can directly verify that the zeroes and poles on the unit circle of the principal minors of $\Phi(z)$ must have even multiplicity.
 Now, by setting $\mathbb{T}:=\{\,z\in\C\,:\,|z|=1 \,\}$, we can apply Lemma \ref{lemma5-pre}. By considering the minors of order $\ell=1$, it follows that $\nu_i^{(1)}$ is even for all $i=1,2,\dots,t$. Similarly, by considering the minors of order $\ell=2$ in Lemma \ref{lemma5-pre}, it follows that  $ \nu_i^{(1)}+\nu_i^{(2)}$ is even for all $i=1,2,\dots,t$. Since $\nu_i^{(1)}$ is even, then also $\nu_i^{(2)}$ must be even for all $i=1,2,\dots,t$. By iterating the argument, we conclude that every zero/pole on the unit circle of the diagonal elements of $D(z)$ has even multiplicity. 
\end{proof}

\begin{remark}
Lemma \ref{lemma5-pre} can also be used to obtain an alternative proof of \cite[Lemma 4, point 2]{Youla-1961}, which represents the continuous-time counterpart of Lemma \ref{lemma5}.
\end{remark}

\begin{lemma}\label{lemma-review}
Let $\Phi(z)\in\R(z)^{n\times n}$ be a spectrum of normal rank $\mathrm{rk}(\Phi)=r$ and let $D(z)\in\R(z)^{r\times r}$ be its Smith-McMillan canonical form. Then $D(z)$ can be written as
\begin{equation}\label{eq:D(z)-decomposition}
D(z) = \Sigma(z)\Lambda^*(z)\Theta^*(z) \Theta(z)\Lambda(z)
\end{equation}
where $\Lambda(z)$ is diagonal, canonic and analytic with its inverse in $\{\,z\in\C\,:\,|z|\geq 1\,\}$ and, if $z=0$ is either a zero, pole or both of $D(z)$, $\Lambda(z)$ has the same structural indices at $z=0$ of $D(z)$; $\Theta(z)$ is diagonal, canonic and analytic with its inverse in $\{\, z\in\C\,:\, |z|\neq 1\,\}$; $\Sigma(z)$ has the form
\begin{equation}\label{eq:sigmadelta-2}
\Sigma(z)=\mathrm{diag}\left[e_1(z),e_2(z),\dots,e_r(z)\right],
\end{equation} 
with $e_i(z)=\alpha_{i} z^{k_i},\, \alpha_{i}\in \R_0,\, k_i\in\mathbb{Z},\, i=1,\dots,r$.
\end{lemma}
\begin{proof}
By direct computation, we obtain
\begin{equation}\label{eq:sigmabarD}
D^*(z)=\Sigma'(z)\bar{D}(z),
\end{equation}
where $\bar{D}(z)$ is canonic and $\Sigma'(z)$ is a diagonal matrix with elements $\alpha z^k$, $\alpha\in \R_0$, $k\in\mathbb{Z}$, on its diagonal.
Since $\Phi(z)$ is a spectrum, we can write
\[\Phi(z)=C(z)D(z)F(z)=F^*(z) D^*(z) C^*(z) =\Phi^*(z),\]
The matrices $F(z)\in\R[z]^{n\times r}$, $C(z)\in\R[z]^{r\times n}$, are unimodular, while  $F^*(z)$, $C^*(z)$ are L-unimodular. By Lemma \ref{lemma1},  $F(z)$, $C(z)$, $F^*(z)$, $C^*(z)$ are analytic in $\C_0$ with their inverses. Thus, we have (see \cite[Ch.6, \S5, Ex.6]{Kailath-1998}) 
\[
\min_{\mathbf{i}} v_\alpha(d_{\mathbf{ii}}^{(\ell)})=\min_{\mathbf{i}} v_\alpha({d^{*(\ell)}_{\mathbf{ii}}}),\ \ \forall \alpha\in\C_0,\ \forall \ell : 1\leq \ell\leq r,
\] 
where $d_{\mathbf{ii}}^{(\ell)}$ and ${d^{*(\ell)}_{\mathbf{ii}}}$ denote the $\ell\times \ell$ minors of $D(z)$ and $D^*(z)$, respectively, obtained by selecting those rows and columns whose indices appear in the ordered $\ell$-tuple $\mathbf{i}$. The previous equation implies that, for every $\alpha\in \C_0$, being either a pole, zero or both of $D(z)$, $D^*(z)$ has the same structural indices at $\alpha$ of $D(z)$. Therefore, since by (\ref{eq:sigmabarD}) $\bar{D}(z)$ is canonic, it follows that
\[
D^*(z)=\Sigma''(z)  D(z)
\]
where $\Sigma''(z)$  is diagonal with elements $\alpha z^k$, $\alpha\in \R_0$, $k\in\mathbb{Z}$, on its diagonal.
This means that any zero/pole at $\alpha\in\C_0$ in the diagonal terms of $D(z)$ is accompanied by a zero/pole at $\alpha^{-1}$, and we can always write $D(z)$ as
\begin{align}\label{eq:Dzdec}
D(z)=\Sigma_1(z)\Lambda^*(z) \Delta(z)\Lambda(z),
\end{align}
where $\Sigma_1(z)$ is diagonal with elements $ \alpha z^{k}$, $\alpha\in\R_0$, $k\in\mathbb{Z}$, on its diagonal; $\Lambda(z)$ and $\Delta(z)$ are diagonal, canonic and analytic with their inverse in $\{\,z\in\C\,:\,|z|\geq 1\,\}$ and $\{\, z\in\C\,:\, |z|\neq 1\,\}$, respectively. Moreover, if $z=0$ is either a pole, zero or both of $D(z)$, $\Lambda(z)$ possesses the same structural indices at $z=0$ of $D(z)$. As a matter of fact, let $\alpha_{i,k}$, $i=1,\dots,p_k$, and $\beta_{j,k}$, $j=1,\dots,q_k$, be the zeroes and poles, respectively, in $\{\,z\in\C_0\,:\,|z|< 1\,\}$ of $[D(z)]_{kk}$ and let $h_k\in \mathbb{Z}$ be the valuation at $z=0$ of $[D(z)]_{kk}$. We can write, for all $k=1,\dots,r$,
\begin{align*}
&[D(z)]_{kk}=z^{h_k}\frac{\prod_{i=1}^{p_k}(z-\alpha_{i,k}^{-1})(z-\alpha_{i,k})}{\prod_{j=1}^{q_k}(z-\beta_{j,k}^{-1})(z-\beta_{j,k})}[\Delta(z)]_{kk}\\
&= \underbrace{\gamma_k \frac{z^{h_k}}{z^{q_k-p_k}}}_{[\Sigma_1(z)]_{kk}} \underbrace{z^{-h_k}	\frac{\prod_{i=1}^{p_k}(z^{-1}-\alpha_{i,k})}{\prod_{j=1}^{q_k}(z^{-1}-\beta_{j,k})}}_{[\Lambda^*(z)]_{kk}}[\Delta(z)]_{kk}\,\cdot \\
&\hspace{5cm} \cdot \underbrace{z^{h_k}	\frac{\prod_{i=1}^{p_k}(z-\alpha_{i,k})}{\prod_{j=1}^{q_k}(z-\beta_{j,k})}}_{[\Lambda(z)]_{kk}} 
\end{align*}
with $\gamma_k:=(-1)^{q_k-p_k}\frac{\prod_{j=1}^{q_k}\beta_{j,k}}{\prod_{i=1}^{p_k}\alpha_{i,k}}$. 

Now, by exploiting Lemma \ref{lemma5}, $\Delta(z)$ can be written as
\[
\Delta(z)=\Theta^2(z)=\Sigma_2(z)\Theta^*(z) \Theta(z),
\]
with $\Sigma_1(z)$ diagonal with elements $ \pm z^{k}$,  $k\in\mathbb{Z}$, on its diagonal and $\Theta(z)$ diagonal, canonic and analytic together with its inverse in $\{\,z\in\C\,:\,|z|\neq 1\,\}$. Finally, we can rearrange $D(z)$ in the form
\[
D(z)=\Sigma(z)\Lambda^*(z)\Theta^*(z) \Theta(z)\Lambda(z),
\]
where $\Sigma(z):=\Sigma_1(z)\Sigma_2(z)$ has the form in (\ref{eq:sigmadelta-2}).
\end{proof}

To conclude this section, we report below another useful result.

\begin{lemma}\label{lemma7}
Let $\Psi(z)\in\R[z,z^{-1}]^{r\times r}$ be a  para-Hermitian L-unimodular matrix which is positive definite on the unit circle. Then,  $\Psi^{\rm hc}$ is nonsingular if and only if $\Psi(z)$ is a constant matrix.
\end{lemma}
\begin{proof}
If $\Psi(z)$ is a constant matrix then $\Psi^{\rm hc}=\Psi(z)$ is nonsingular, by definition of L-unimodular matrix.

Conversely, assume that $\Psi^{\rm hc}$ is nonsingular. Let us denote by $K_i\in\mathbb{Z}$, $i=1,\dots,r$, the maximum-degree of the $i$-th column of $\Psi(z)$ and by $k_i\in\mathbb{Z}$, $i=1,\dots,r$, the minimum-degree of the $i$-th row of $\Psi(z)$.
 Since $\Psi(z)=\Psi^*(z)$, we have that $\det \Psi(z)$ is a nonzero real constant and
\begin{equation}\label{eq:Ki-ki}
K_i=-k_i, \quad i=1,\dots,r.
\end{equation}
Moreover, since $\Psi(z)$ is positive definite on the unit circle, the diagonal elements of $\Psi(z)$ cannot be equal to zero and, therefore, $K_i\geq 0$, $i=1,\dots,r$.
Actually, the nonsingularity of $\Psi^{\rm hc}$ yields
\begin{equation}\label{eq:Kizero}
K_i=0,\quad  i=1,\dots,r,
\end{equation} 
otherwise one can check, by exploiting the Leibniz formula for determinants, that the maximum-degree of $\det \Psi(z)$ would be strictly positive; but this is not possible since, as noticed above, $\det \Psi(z)$ is a nonzero real constant and so $\maxdeg (\det \Psi(z)) = 0$.
By (\ref{eq:Kizero}), all the entries of $\Psi(z)$ must have maximum-degree less than or equal to zero. 
But, by (\ref{eq:Ki-ki}), $k_i=-K_i$ for all $i=1,\dots,r$, and so (\ref{eq:Kizero}) also implies that all the entries of $\Psi(z)$ must have minimum-degree greater than or equal to zero. We conclude that 
\[
\maxdeg [\Psi(z)]_{ij}=\mindeg [\Psi(z)]_{ij}=0, \quad i,\, j=1,\dots,r,
\]
and, therefore, $\Psi(z)$ must be a constant matrix.
\end{proof}

\section{Proof of the main theorem}\label{sec:main-theorem}

We are now ready to prove our main result.
For the sake of clarity and readability, we first prove the special case of 
Theorem \ref{thmsf-dt} and we then proceed to the proof of our general Theorem \ref{thmsf-dt-g}.

{\bf \em Proof of Theorem \ref{thmsf-dt}:}
We first prove statement \ref{item:thmsf-dt(iii)}). Let $W(z)$ and $W_1(z)$ be two matrices satisfying \ref{item:thmsf-dt(i)}) and \ref{item:thmsf-dt(ii)}). Then,
\begin{align}\label{eq:thmsf-dt-1}
W^*(z) W(z)=W_1^*(z) W_1(z).
\end{align}  
The latter equation implies $V^*(z) V(z)=I_r$, where $V(z):=W_1(z)W^{-R}(z)$ is analytic in $\{\,z\in\overline{\C}\,:\,|z|>1 \,\}$. Thus,  $V(z)\in\R(z)^{r\times r}$ is a para-unitary matrix  analytic in $\{\, z\in\overline{\C} \,:\, |z|> 1\,\}$. Moreover,
we have that $\Delta_1(z):=W_1(z)-V(z)W(z)=W_1(z)[I_n-W^{-R}(z)W(z)]$
satisfies
\begin{align}\label{eq-unicita}
&\Delta_1^*(z)\Delta_1(z) =\nonumber \\
& =  [I_n-W^*(z)W^{-R*}(z)]W_1^*(z)W_1(z)[I_n-W^{-R}(z)W(z)]\nonumber \\
& =  [I_n-W^*(z)W^{-R*}(z)]W^*(z)W(z)[I_n-W^{-R}(z)W(z)]\nonumber \\
& =   0,
\end{align}
so that 
\begin{equation}\label{w1=vw}
W_1(z)=V(z)W(z)
\end{equation}
yielding that
$
V^{-1}(z)=W(z)W_1^{-R}(z)$ is analytic in $\{\, z\in\overline{\C} \,:\, |z|> 1\,\}$. In view of Lemma \ref{lemma2}, we conclude that $V(z)$ is a constant orthogonal matrix.

Consider now statement \ref{item:thmsf-dt(iv)}) and let $\Phi(z)=L^*(z) L(z)$ where $L(z)\in\R(z)^{n\times r}$ is analytic in $\{\,z\in\overline{\C}\,:\,|z|>1 \,\}$. In this case, we can write
\[
L^*(z) L(z)=W^*(z) W(z).
\]
The latter equation implies $V^*(z) V(z)=I_r$,
where $V(z):=L(z)W^{-R}(z)$ and $W(z)\in\R(z)^{r\times n}$ is a rational matrix satisfying \ref{item:thmsf-dt(i)}) and \ref{item:thmsf-dt(ii)}). Since $L(z)$ and $W^{-R}(z)$ are both analytic in $\{\,z\in\overline{\C}\,:\,|z|>1 \,\}$, then $V(z)\in\R(z)^{r\times r}$ is a para-unitary matrix analytic in $\{\, z\in\overline{\C} \,:\, |z|> 1\,\}$.
The same computation that led to (\ref{w1=vw}) now gives $L(z)=V(z)W(z)$.

Now, we provide a constructive proof of statements \ref{item:thmsf-dt(i)}) and \ref{item:thmsf-dt(ii)}), which represent the core of the Theorem. The procedure is divided in four steps.

\emph{Step 1.} Reduce $\Phi(z)$ to the Smith-McMillan canonical form. By using the same standard procedure described in \cite[Thm.2]{Youla-1961}, we arrive at 
\begin{align}
\Phi(z)=C(z)D(z)F(z),
\end{align}
where $C(z)\in\mathbb{R}[z]^{n\times r}$, $F(z)\in\mathbb{R}[z]^{r\times n}$ are unimodular polynomial matrices and $D(z)\in\mathbb{R}(z)^{r\times r}$ is diagonal and canonic.

\emph{Step 2.} According to Lemma \ref{lemma-review}, we can write $D(z)$ in the form
\begin{equation}\label{eq:smith-mcmillan-phi}
D(z)=\Sigma(z)\Lambda^*(z)\tilde{\Delta}(z)\Lambda(z),
\end{equation}
where:
\begin{enumerate}
\item $\Lambda(z)\in\mathbb{R}(z)^{r\times r}$ is diagonal, canonic and analytic together with $\Lambda^{-1}(z)$ in $\{\,z\in\C\,:\,|z|\geq 1 \,\}$ and, if $z=0$ is either a pole, zero or both of $D(z)$, $\Lambda(z)$ possesses the same  structural indices at $z=0$ of $D(z)$;
\item $\tilde{\Delta}(z):=\Theta^*(z) \Theta(z)=\tilde{\Delta}^*(z)$, where $\Theta(z)\in\R(z)^{r\times r}$ is diagonal, canonic and analytic together with $\Theta^{-1}(z)$ in $\{\,z\in\C\,:\,|z|\neq 1 \,\}$;
\item $\Sigma(z)\in\R(z)^{r\times r}$ is diagonal of the form 
\[
\Sigma(z)=\mathrm{diag}\left[e_1(z),e_2(z),\dots,e_r(z)\right],
\]
where $e_i(z)=\alpha_{i} z^{k_i},\, \alpha_{i}\in \mathbb{R}_0,\, k_i\in\mathbb{Z},\, i=1,\dots,r$.
\end{enumerate}
Let us define
\[
 A(z) :=  C(z)\Sigma(z)\Lambda^*(z),\quad B(z) :=  \Lambda(z)F(z).
\]
We have that $\Phi(z)=A(z)\tilde{\Delta}(z)B(z)$
is a left-standard factorization of $\Phi(z)$.

{\emph{Step 3.}}
Let  $I(z):=B^{-R}(z)\Theta^{-1}(z)$.
By (\ref{eq:A(z)B(z)}), we have $A^*(z)=N(z)B(z)$ and, therefore,
\begin{align}\label{eq:Mtilde(z)-dt}
&I^*(z)\Phi(z) I(z)= I^*(z) \Phi^*(z) I(z)\nonumber \\
& =  \Theta^{-*}(z)B^{-R*}(z) B^*(z)\tilde{\Delta}^*(z) N(z) B(z) B^{-R}(z)\Theta^{-1}(z)\nonumber\\
& =   \Theta^{-*}(z) \Theta^*(z) \Theta(z) N(z) \Theta^{-1}(z)\nonumber \\
& =  \Theta(z) N(z) \Theta^{-1}(z)=:\Psi(z),
\end{align}
where $N(z)=A^*(z)B^{-R}(z)\in\mathbb{R}[z,z^{-1}]^{r\times r}$ is a L-unimodular matrix. By (\ref{eq:Mtilde(z)-dt}), $\Psi(z)$ is a para-Hermitian matrix positive semi-definite   on the unit circle. Actually a good deal more is true. We notice that $A(z)\tilde{\Delta}(z)B(z)$ and $B^*(z)\tilde{\Delta}(z) A^*(z)$ are two left-standard factorizations of $\Phi(z)$. Hence, by replacing $\Delta_1(z)$ with $\tilde \Delta(z)=\tilde \Delta^*(z)$ in (\ref{eq:M(z)Delta(z)N(z)}), we obtain
\begin{align}\label{eq:Delta(z)N(z)Delta(z)}
\tilde \Delta(z) N(z) \tilde \Delta^{-1}(z)= M(z),
\end{align}
where $M(z)\in\mathbb{R}[z,z^{-1}]$ is L-unimodular. Since $\tilde \Delta(z)=\Theta^*(z) \Theta(z)$ is diagonal and \[\Theta(z):=\mathrm{diag}[\theta_1(z),\dots,\theta_r(z)]\] is canonic, (\ref{eq:Delta(z)N(z)Delta(z)}) implies that $[N(z)]_{ij}$ is divisible by the L-polynomial $[\tilde \Delta(z)]_{jj}/[\tilde \Delta(z)]_{ii}, \ j\geq i$. But 
\begin{align}
[\tilde \Delta(z)]_{ii}&=\theta_i^*(z) \theta_i(z)=\theta_i(1/z) \theta_i(z)=\pm z^{k_{i}} \theta_i^2(z),\nonumber
\end{align} 
where  
 $k_i\in\mathbb{Z}, \ i=1,\dots,r$. Hence, $[N(z)]_{ij}$ must be divisible by the polynomial
\[
f_{ij}^2(z):=\frac{\theta_j^2(z)}{\theta_i^2(z)}, \ \ \ j\geq i,
\]
and, a fortiori, by 
\[
f_{ij}(z)= \frac{\theta_j(z)}{\theta_i(z)}, \ \ \ j\geq i.
\] 
This suffices to establish that $\Psi (z)$ is L-polynomial. Actually, by (\ref{eq:Mtilde(z)-dt}), it follows that $\Psi(z)$ has determinant which is a real nonzero constant. Hence, $\Psi(z)$ is L-unimodular and positive definite on the unit circle.
The problem is now reduced to that of finding a factorization of $\Psi(z)$ of the form
\begin{align}
\Psi(z)=P^*(z) P(z),
\end{align}
where $P(z)\in\mathbb{R}[z]^{r\times r}$ is a unimodular  polynomial matrix.
After this is achieved, the desired factorization for $\Phi(z)$ is obtained as $\Phi(z)=W^*(z) W(z)$ with
\begin{align}\label{eq:H(z)-dt}
W(z):&=  P(z)\Theta(z)B(z)\nonumber \\
&=  P(z)\Theta(z)\Lambda(z)F(z)\nonumber \\
&=  P(z)D_+(z)F(z),
\end{align}
where we have defined $D_+(z):=\Theta(z)\Lambda(z)$.
Indeed, by straightforward algebra,
\begin{align}
W^*(z) W(z) & =  B^*(z) \Theta^*(z) P^*(z)  P(z)\Theta(z)B(z)\nonumber \\
& =  B^*(z) \tilde{\Delta}(z) N(z)B(z)\nonumber \\
& =  B^*(z) \tilde{\Delta}(z) A^*(z)\nonumber \\
& =  \Phi^*(z) = \Phi(z).\nonumber
\end{align} 

\emph{Step 4.} We illustrate an algorithm which provides a factorization of a para-Hermitian L-unimodular polynomial matrix $\Psi(z)=\Psi^*(z)\in\mathbb{R}[z,z^{-1}]^{r\times r}$ positive definite on the unit circle into the product $P^*(z) P(z)$, where $P(z)$ is a unimodular polynomial matrix. 

The algorithm consists of the following two steps.  First of all, we define $\Psi_1(z):=\Psi(z)$ and denote by $h\in\mathbb{N}$ the loop counter of the algorithm, which is initially set to $h:=1$. 

\begin{enumerate}
\item \label{item:thmsf-dt-step4I}
Let $K_i\in\mathbb{Z}, \ i=1,\dots,r$, be the maximum-degree of the $i$-th column of $\Psi_h(z)$ and $k_i\in\mathbb{Z}, \ i=1,\dots,r$, be the minimum-degree of the $i$-th row of $\Psi_h(z)$.
Consider the {highest-column-degree coefficient matrix} of $\Psi_h(z)$, denoted by $\Psi_h^\mathrm{hc}$, and the lowest-row-degree coefficient matrix of $\Psi_h(z)$, denoted by $\Psi_h^\mathrm{lr}$. As noticed in the proof of Lemma \ref{lemma7}, the positive nature of $\Psi_h(z)$ implies that $K_i\geq 0$ for all $i= 1,\dots,r$. Moreover, the para-Hermitianity of $\Psi_h(z)$ implies that $\Psi_h^\mathrm{hc}=(\Psi_h^\mathrm{lr})^\top$
which, in turn, yields $K_i=-k_i$ for all $i=1,\dots,r$.

By Lemma \ref{lemma7}, it follows that $\Psi_h^{\mathrm{hc}}$ is nonsingular if and only if $\Psi_h(z)$ is a constant matrix.
If $\Psi_h(z)$ is a constant matrix, we set $\bar{h}:=h$ and go to step \ref{item:thmsf-dt-step4II}). If this is not the case, we calculate a nonzero vector $\mathbf{v}_h:=[v_1 \ v_2 \ \dots \ v_r]^\top\in\R^r$ such that $\Psi_h^\mathrm{hc}\mathbf{v}_h=\mathbf{0}$. 
Let us define the \emph{active index set}
\[
\mathcal{I}_h:=\{\,i \, :\, v_i\neq 0\,\}
\] 
and the \emph{highest maximum-degree active index set},  $\mathcal{M}_h\subset \mathcal{I}_h$,
\[
\mathcal{M}_h:=\{\,i\in\mathcal{I}_h \, :\, \ K_i\geq K_j, \ \forall\, j\in\mathcal{I}_h\,\}.
\]
We pick an index $p\in\mathcal{M}_h$.
Then, we define the polynomial matrix
\begin{align}
&\qquad \qquad \quad \quad \quad \quad \ \ \ \ {\scriptsize \text{column $p$}}  \nonumber \\
\Omega_h^{-1}(z)&:=\left[\begin{smallarray}{ccccccc}
\ 1 \ & \cdots & 0 & \frac{v_1}{v_p}z^{K_p-K_1} & 0 & \cdots & 0 \\
0 &\ \ddots \ &  & \vdots & &  & 0 \\
\vdots   &  &\ 1 \ & \frac{v_{p-1}}{v_p}z^{K_p-K_{p-1}} & & & \vdots \\
\vdots   &  &  &\ 1\ & & & \vdots \\
\vdots   &  &  & \frac{v_{p+1}}{v_p}z^{K_p-K_{p+1}} &\ 1 \ & & \vdots \\
0   &  &  & \vdots & &\ \ddots \ & 0 \vspace{0.27cm} \\ 
0  & \cdots & 0 & \frac{v_r}{v_p}z^{K_p-K_r} & 0 & \cdots & \ 1 \  \end{smallarray}\right].\nonumber\\
\label{eq:matrix-reduction-dt}
\end{align}
Notice that the entry at $(i,p)$ of $\Omega_h^{-1}(z)$ has the form 
\begin{align}\label{eq:alpha-delta}
\frac{v_i}{v_p}z^{K_p-K_i}=\alpha_iz^{\delta_i} , \ \ \ i=1,\dots,r,
\end{align} 
with $\alpha_i:={v_i}/{v_p}\in\R$ and $\delta_i:=K_p-K_i\geq 0$. In fact, if $K_i>K_p$, then $v_i=0$ and so $\alpha_i=0$. By (\ref{eq:matrix-reduction-dt}), $\det \Omega_h^{-1}(z)=1$ and, therefore, $\Omega_h^{-1}(z)\in\mathbb{R}[z]^{r\times r}$ is a unimodular polynomial matrix. By operating the transformation
\[
\Psi_{h+1}(z):=\Omega_h^{-*}(z) \Psi_h(z) \Omega_h^{-1}(z),
\]
we obtain a new positive definite matrix $\Psi_{h+1}(z)$ with the same determinant of $\Psi_h(z)$. Furthermore, the maximum-degree of the $p$-th column of $\Psi_{h+1}(z)$ is lower than $K_p$, while the maximum-degree of the $i$-th column, $i\neq p$, is not greater than $K_i$.

This fact needs a detailed explanation. If we post-multiply $\Psi_h(z)$ by $\Omega_h^{-1}(z)$, we obtain a matrix of the form
\begin{align}
&\Psi_h'(z):=\Psi_h(z)\Omega_h^{-1}(z)\nonumber\\
&=\left[\begin{array}{c|c|c} [\Psi_h(z)]_{1:r,1:p-1} & \boldsymbol{\psi}_{h}(z) & [\Psi_h(z)]_{1:r,p+1:r} \\
\end{array}\right],\nonumber
\end{align} 
where all the L-polynomials in the $p$-th column vector \begin{align}\label{eq:psi-vector}
\boldsymbol{\psi}_h(z)=[\Psi_h(z)]_{1:r,p:p}+ \sum_{i\neq p} \alpha_iz^{\delta_i} [\Psi_h(z)]_{1:r,i:i}
\end{align}
have maximum-degree lower than $K_p$, since $\Psi_h^\mathrm{hc}\mathbf{v}_h=\mathbf{0}$,
and minimum-degree which satisfies
\begin{align}\label{eq:min-deg-ki}
\mindeg [\boldsymbol{\psi}_h(z)]_i\geq k_i=-K_i, \ \ \ i=1,\dots,r,
\end{align}
since in (\ref{eq:psi-vector}) $\delta_i\geq 0$, for all $i$ such that $\alpha_i\neq 0$. Now, by pre-multiplying $\Psi_h'(z)$ by $\Omega_h^{-*}(z)$, the resulting matrix $\Psi_{h+1}(z)$ can be written in the form
\begin{align}\label{eq:psihc}
&\Psi_{h+1}(z)=\Omega_h^{-*}(z) \Psi_h(z) \Omega_h^{-1}(z)\nonumber\\
&=\left[\begin{smallarray}{c|c|c} [\Psi_h(z)]_{1:p-1,1:p-1} & \boldsymbol{\psi}_{h+1}'(z) & [\Psi_h(z)]_{1:p-1,p+1:r} \\
\hline
\boldsymbol{\psi}'^\top_{h+1}(z^{-1}) & \psi_{h+1}''(z) & \boldsymbol{\psi}'''^\top_{h+1}(z^{-1}) \\ 
\hline
[\Psi_h(z)]_{p+1:r,1:p-1} & \boldsymbol{\psi}_{h+1}'''(z) & [\Psi_h(z)]_{p+1:r,p+1:r}\end{smallarray}\right],\nonumber
\end{align}
where the $p$-th column vector \[\left[\begin{array}{c|c|c}\boldsymbol{\psi}'^\top_{h+1}(z) & \psi_{h+1}''(z) & \boldsymbol{\psi}'''^\top_{h+1}(z) \end{array}\right]^\top\] differs from $\boldsymbol{\psi}_h(z)$ only for the value of the $p$-th entry $\psi_{h+1}''(z)$. Moreover, the maximum-degree of  $\psi_{h+1}''(z)$ cannot increase after the operation is performed, since
\[
\psi_{h+1}''(z)=[\boldsymbol{\psi}_h(z)]_p +\sum_{i\neq p}\alpha_i z^{-\delta_i}[\boldsymbol{\psi}_h(z)]_i,
\] 
and, by (\ref{eq:alpha-delta}), $\delta_i\geq 0$, for all $i$ such that $\alpha_i\neq 0$. We conclude that all the L-polynomials in the $p$-th column of $\Psi_{h+1}(z)$ have maximum-degree lower than $K_p$, while, by (\ref{eq:min-deg-ki}), the maximum-degree of all  the other columns does not increase. We notice also that, since $\Psi_{h+1}(z)=\Psi_{h+1}^*(z)$, all the L-polynomials in the $p$-th row of $\Psi_{h+1}(z)$ have minimum-degree greater than $k_p=-K_p$,  while the minimum-degree of all the other rows does not decrease. Eventually, we update the value of the loop counter $h$ by setting $h:=h+1$ and return to step \ref{item:thmsf-dt-step4I}).

\item \label{item:thmsf-dt-step4II} 
Since $\Psi_{\bar{h}}\in\mathbb{R}^{r\times r}$ is positive definite, we can always factorize it into the product $\Psi_{\bar{h}}=C^\top C$ where $C\in\mathbb{R}^{r\times r}$, by using standard techniques such as the Cholesky decomposition (see, e.g., \cite[Ch.4]{Golub-1996}). Finally, we have constructed a polynomial unimodular matrix
\[
P(z)=C \Omega_{\bar{h}-1}(z)\Omega_{\bar{h}-2}(z)\cdots \Omega_1(z).
\]
such that $\Psi(z)=P^*(z) P(z)$.
\end{enumerate}

It is worthwhile noticing that the iterative procedure of step \ref{item:thmsf-dt-step4I}) is always brought to an end (after a maximum of $K_1+\dots+K_p$ iterations) since at the $h$-th iteration the maximum-degree of a column of $\Psi_h(z)$ is reduced at least by one, while the maximum-degree of all the other columns does not increase.

To complete the proof of statements \ref{item:thmsf-dt(i)}) and \ref{item:thmsf-dt(ii)}), we notice that, by construction, the rational matrix $W(z)$, as defined in (\ref{eq:H(z)-dt}), and its right inverse are analytic in $\{\,z\in\C\,:\,|z|>1 \,\}$. Moreover, we recall that, if $z=0$ is either a pole, zero or both of $D(z)$, $D_+(z)$ and $D(z)$ have the same zero-pole structure at $z=0$. Now, suppose, by contradiction, that $W(z)$ has a pole at $z=\infty$. Then $W^*(z)$ has a pole at $z=0$. But, since $\Phi(z)=W^*(z)W(z)$, it follows that
\begin{align}\label{eq:Wstar-infty}
W^*(z)&= \Phi(z)W^{-R}(z) \nonumber\\
&=C(z)D(z)F(z)F^{-R}(z)D_+^{-1}(z)P^{-1}(z) \nonumber\\
&=C(z)D(z)D_+^{-1}(z) P^{-1}(z)\nonumber\\
&=C(z)D_-(z)P^{-1}(z),
\end{align}
where $D_{-}(z):= D(z)D_+^{-1}(z)$ has no pole at $z=0$. Since $P^{-1}(z)$ and $C(z)$ are unimodular matrices, in view of (\ref{eq:Wstar-infty}), also $W^*(z)$ has no pole at $z=0$. Hence, the contradiction. We conclude that $W(z)$ has no pole at infinity. Finally, by following a similar argument, it can be verified that also $W^{-R}(z)$ has no pole at infinity.

Now consider statement \ref{item:thmsf-dt(v)}). If $\Phi(z)$ is analytic on the unit circle, then $\Theta(z)$ does not possess any finite pole. This, in turn, implies that $D_+(z)=\Theta(z)\Lambda(z)$ is analytic in $\{\, z\in\overline{\C}\,:\,|z|\geq 1 \}$. Thus, $W(z)$, as defined in (\ref{eq:H(z)-dt}), is also analytic in the same region. 

As for point \ref{item:thmsf-dt(vi)}), the additional assumption that the rank of $\Phi(z)$ is constant on the unit circle implies that $\Theta(z)$ does not possess any finite zero. Thus, $\Theta(z)=I_r$ and, by (\ref{eq:H(z)-dt}),
\[
W^{-R}(z)=F^{-R}(z) \Lambda^{-1}(z) P^{-1}(z)
\]
is analytic in $\{\, z\in\overline{\C}\,:\,|z|\geq 1 \}$. Hence, $W(z)$ and its right inverse $W^{-R}(z)$ are both analytic in $\{\, z\in\overline{\C}\,:\,|z|\geq 1 \}$.

Lastly, consider point \ref{item:thmsf-dt(vii)}). As shown in (\ref{eq:smith-mcmillan-phi}), the Smith-McMillan canonical form of $\Phi(z)$, $D(z)$, is connected to that of $W(z)$, $D_+(z)=\Theta(z)\Lambda(z)$, by
\begin{align}\label{eq:D(z)-plus-dt}
D(z)=\Sigma(z) D_+^*(z)D_+(z),
\end{align}
where $\Sigma(z)\in\mathbb{R}(z)^{r\times r}$ is a diagonal matrix with elements $ \alpha_{i} z^{k_i}$, $\alpha_{i}\in\R_0$, $k_i\in\mathbb{Z}$, on its diagonal. 
Let $p_1,\dots,p_h$ be the nonzero finite poles of $\Phi(z)$. By (\ref{eq:D(z)-plus-dt}), it follows that
\begin{equation}\label{MMd=Sod}
\delta(\Phi;p_i)=\begin{cases} \delta(W;p_i) & \text{if } | p_i|<1, \\
2\delta(W;p_i) & \text{if } |p_i|=1,\\
\delta(W;1/p_i) & \text{if } |p_i|>1. \end{cases}
\end{equation}
Moreover, if $p\in\overline{\C}$ is a pole of $\Phi(z)$ of degree $\delta(\Phi;p)$ then also $1/p$ is a pole of $\Phi(z)$ of the same degree and if $p\in\overline{\C}$ is not a pole of $\Phi(z)$ then neither $p$ nor $1/p$ are poles of $W(z)$.
Thus, we have
\begin{align}\label{eq:delta-pii-z}
\sum_{i=1}^h\delta(\Phi;p_i)&=\sum_{\substack{i\, :\, |p_i|<1}} \delta(W;p_i)\, + \sum_{\substack{i\, :\, |p_i|>1}}  \delta(W;1/p_i)\ + \nonumber\\ 
&\hspace{2.85cm} +\sum_{\substack{i\, :\, |p_i|=1}} 2\delta(W;p_i)\nonumber \\
&= 2\sum_{\substack{i\, :\, |p_i|\leq 1}} \delta(W;p_i) 
\end{align} 
By (\ref{eq:mcmillan-degree}), the McMillan degree of a rational matrix equals the sum of the degrees of all its poles, including the pole at infinity. If $\Phi(z)$ has no pole at infinity, then (\ref{eq:delta-pii-z}) directly yields $\delta_M(\Phi)=2\delta_M(W)$. 
Otherwise, assume that $\Phi(z)$ has a pole at infinity. Since $W(z)$ and $\Phi(z)$ have the same structural indices at $z=0$ and $W(z)$ has no pole at $z=\infty$, it follows that
\begin{align}\label{eq:delta-zero}
\delta(\Phi;\infty)=\delta(\Phi;0)=\delta(W;0)\ \ \ \ \text{and} \ \ \ \ \delta(W;\infty)=0.
\end{align}
Therefore, by equations (\ref{eq:delta-pii-z}) and (\ref{eq:delta-zero}),
\begin{align}
\delta_M(\Phi)&=\sum_{i=1}^h\delta(\Phi;p_i)+\delta(\Phi;0)+\delta(\Phi;\infty)\nonumber\\
&=2\sum_{\substack{i\, :\, |p_i|\leq 1}} \delta(W;p_i) +2\delta(W;0)=2\delta_M(W),\nonumber
\end{align}
\endproof

We are now ready to prove our main Theorem \ref{thmsf-dt-g}. Many of the ideas for this proof  can be elaborated from those of the proof of Theorem \ref{thmsf-dt}.

{\bf \em  Proof of Theorem \ref{thmsf-dt-g}:} { We first show  how to modify the  constructive procedure used in the proof of Theorem \ref{thmsf-dt} in order to obtain a spectral factor $W(z)$ which satisfies points 1) and 2).
With reference to step 2 in the proof of Theorem \ref{thmsf-dt}, we rearrange the Smith-McMillan form of $\Phi(z)$ as
\begin{align}\label{eq:Dz-dec-thm1}
D(z)=\Sigma(z)  \Lambda(z)\tilde{\Delta}(z) \Lambda(z),
\end{align}
where the only difference with respect to the decomposition in  (\ref{eq:smith-mcmillan-phi}) is that here  $\Lambda(z)\in\mathbb{R}(z)^{r\times r}$ is diagonal, canonic and analytic in $\mathscr{A}_p\setminus\{\infty\}$ with $\Lambda^{-1}(z)$ analytic in $\mathscr{A}_z\setminus\{\infty\}$. Moreover, if $0\not\in \mathscr{A}_p$ and $z=0$ is a pole of $D(z)$, then $\Lambda(z)$ has the same negative structural indices at $z=0$ of $\Phi(z)$, and if $0\not\in \mathscr{A}_z$ and $z=0$ is a zero of $D(z)$, then $\Lambda(z)$ has the same positive structural indices at $z=0$ of $\Phi(z)$. 

Now, to apply the procedure described in the proof of Theorem \ref{thmsf-dt}, it suffices to prove that for any choice of the unmixed-symplectic sets $\mathscr{A}_p$ and $\mathscr{A}_z$, the para-Hermitian matrix $\Psi(z)$, as defined in (\ref{eq:Mtilde(z)-dt}), is still L-unimodular. With reference to the notation introduced in the proof of Theorem \ref{thmsf-dt}, $\Psi(z)$ can be written as
\begin{align}\label{eq:Xi(z)}
\Psi(z)&= \Theta(z) N(z) \Theta^{-1}(z)\nonumber\\
&=\Theta(z)A^*(z)B^{-R}(z) \Theta^{-1}(z)\nonumber\\
&=\Theta(z)\Lambda(z)\Sigma^*(z)C^*(z)F^{-R}(z)\Lambda^{-1}(z) \Theta^{-1}(z)\nonumber\\
&=\Sigma^*(z)D_+(z)\Xi(z)D_+^{-1}(z),
\end{align}
where we have defined $\Xi(z):=C^*(z)F^{-R}(z)\in\R[z,z^{-1}]^{r\times r}$ which is L-unimodular and whose structure does not depend upon the choice of $\mathscr{A}_p$ and $\mathscr{A}_z$. Moreover, in this case, $D_+(z)=\Theta(z)\Lambda(z)$ is diagonal, canonic and analytic in $\mathscr{A}_p\setminus\{\infty\}$ with inverse  analytic in $\mathscr{A}_z\setminus\{\infty\}$. Let us first consider the standard choice $\mathscr{A}_p=\mathscr{A}_z=\{\, z\in\overline{\C}\,:\,|z|>1 \,\}$. In the proof of Theorem \ref{thmsf-dt}, we have shown that $\Psi(z)$ is L-unimodular. Since $D_+(z)$ is diagonal and canonic and $\Sigma^*(z)$ is L-unimodular, by (\ref{eq:Xi(z)}), it follows that $[\Xi(z)]_{ij}\in\R[z,z^{-1}]$ must be divisible  (the concept of divisibility here is the one associated to the ring of L-polynomials) by the polynomial 
\[
p_{ij}(z):=\frac{[D_+(z)]_{jj}}{[D_+(z)]_{ii}}, \ \ \ j\geq i.
\]
On the other hand, let us consider the opposite choice $\mathscr{A}_p=\mathscr{A}_z=\{\, z\in{\C}\,:\,|z|<1 \,\}$. 
By using the right-standard counterpart of Lemma \ref{lemma4} and by following verbatim the argument used in  step 3 of Theorem \ref{thmsf-dt}, it can be proven that $\Psi(z)$ is still L-unimodular. Hence, by (\ref{eq:Xi(z)}),  $[\Xi(s)]_{ij}$ must be also divisible by the L-polynomial $p_{ij}(z^{-1})$, $j\geq i$. Therefore, $[\Xi(z)]_{ij}$ must be divisible by the L-polynomial 
 \[
 q_{ij}(z):= p_{ij}(z) p_{ij}(z^{-1}), \ \ \ j\geq i.
 \]
  Since, for any choice of the unmixed-symplectic sets $\mathscr{A}_p$ and $\mathscr{A}_z$, the factors of $[D_+(z)]_{jj}[D_+(z)]_{ii}^{-1}$, $j\geq i$, are contained in the ones of $q_{ij}(z)$, then $[\Xi(z)]_{ij}$ must be divisible by the polynomial $[D_+(z)]_{jj}[D_+(z)]_{ii}^{-1}$, $j\geq i$, for any choice of $\mathscr{A}_p$ and $\mathscr{A}_z$. We conclude that $\Psi(z)$ must be a L-polynomial matrix for any choice of $\mathscr{A}_p$ and $\mathscr{A}_z$. But, since $\Psi(z)$ is para-Hermitian, $\det \Psi(z)$ is a real constant, hence $\Psi(z)$ is L-unimodular.

To prove point \ref{item:thmsf-dt-g(3)}) we need to show that the McMillan degree of the spectral factor $W(z)$ just obtained equals one half of the McMillan degree of $\Phi(z)$. To this aim, we can follow the same lines of the proof of point \ref{item:thmsf-dt(vii)}) of Theorem \ref{thmsf-dt}.
In fact, we can define $\mathscr{A}_{p,1}:=\mathscr{A}_p\setminus \left(\{\,z\in\C\,:\, |z|=1\,\}\cup\{0,\infty\}\right)$ and partition $\C_0$ as
$$
\C_0=\{\,z\in\C\,:\, 1/z\in\mathscr{A}_{p,1}\,\}\cup \{\,z\in\C\,:\, |z|=1\,\}\cup \mathscr{A}_{p,1}
$$
and replace equation (\ref{MMd=Sod}) with the more general expression for the degree of the pole $p_i$ of $\Phi(z)$
\begin{equation}\label{MMd=Sod-gen}
\delta(\Phi;p_i)=\begin{cases} \delta(W;p_i) & \text{if } 1/p_i\in \mathscr{A}_{p,1}, \\
2\delta(W;p_i) & \text{if } |p_i|=1,\\
\delta(W;1/p_i) & \text{if } p_i\in \mathscr{A}_{p,1}. \end{cases}\nonumber
\end{equation}
The rest of the proof remains the same. 

The proof of point \ref{item:thmsf-dt-g(4)}) is very similar to that of point \ref{item:thmsf-dt(iii)}) of Theorem \ref{thmsf-dt}. The only difference is that the para-unitary matrix function $V(z):=W_1(z)W^{-R}(z)$ and its inverse are not analytic in $\{\, z\in\overline{\C} \,:\, |z|> 1\,\}$ but they are  analytic  in $\mathscr{A}_p$, so that Lemma \ref{lemma2} still applies.

As for point \ref{item:thmsf-dt-g(6)}), we define $V(z):=L(z)W^{-R}(z)$ which is clearly para-unitary and  analytic in  $\mathscr{A}_z$, and the same computation that led to (\ref{w1=vw}), gives $L(z)=V(z)W(z)$.
On the other hand, if $V(z)$ is para-unitary and  analytic in $\mathscr{A}_p$, then it is immediate to check that $L(z):=V(z)W(z)$ is a spectral factor of $\Phi(z)$ and is analytic in  $\mathscr{A}_p$ as well.

The proof of points \ref{item:thmsf-dt-g(7)}) and \ref{item:thmsf-dt-g(8)}) is exactly the same as that of points \ref{item:thmsf-dt(v)}) and \ref{item:thmsf-dt(vi)})
of Theorem \ref{thmsf-dt}.
\endproof}

\subsection{Corollaries}

To conclude this section, we present two straightforward corollaries of Theorem \ref{thmsf-dt}.
The first is a complete parametrization of the set of all spectral factors of a given spectrum.

\begin{corollary}
Let $\Phi(z)$ be a given spectrum and $W(z)$ be any  spectral factor satisfying conditions \ref{item:thmsf-dt(i)}) and \ref{item:thmsf-dt(ii)})  of Theorem \ref{thmsf-dt}.
Let $L(z)\in\R(z)^{m\times n}$, then  $\Phi(z)=L^*(z) L(z)$ if and only if
\[
L(z)=V(z)\left[\begin{array}{c}
I_r  \\
\hline
\mathbf{0}_{m-r, r}
\end{array}\right] W(z),
\]
where $V(z)\in\R(z)^{m\times m}$ is an arbitrary para-unitary matrix and $r=\mathrm{rk}(\Phi)$.
\end{corollary}

\begin{proof}
By repeating an argument used in points \ref{item:thmsf-dt(iii)}) and \ref{item:thmsf-dt(iv)}) of Theorem \ref{thmsf-dt}, we have that $L(z)=U(z)W(z)$,  with $U(z)\in\R(z)^{m\times r}$ a rational matrix satisfying $U^*(z)U(z)=I_r$. 
 If we choose $V(z)\in\R(z)^{m\times m}$ to be any para-unitary matrix with $U(z)$ incorporated into its first $r$ columns, i.e.,
\[
U(z)=V(z)\left[\begin{array}{c}
I_r  \\
\hline
\mathbf{0}_{m-r, r}
\end{array}\right],
\]
we conclude.
\end{proof}

The next result characterizes the spectral factors of  L-polynomial spectra.  
\begin{corollary}
Let $\Phi(z)$ be a spectrum and $W(z)$ be   the spectral factor provided in the (constructive) proof of Theorem \ref{thmsf-dt-g}. Assume that $\Phi(z)$ is L-polynomial.
If $\infty\in\mathscr{A}_p$, then $W(z)$ is polynomial in $z^{-1}$ (so that $W^*(z)$ is polynomial in $z$). Otherwise, $0\in\mathscr{A}_p$ and $W(z)$ is polynomial in $z$ (so that $W^*(z)$ is polynomial in $z^{-1}$).
\end{corollary}
\begin{proof} We consider only the case of $\infty\in\mathscr{A}_p$, the other being similar.
If $\Phi(z)$ is L-polynomial, then the only finite pole it may possess is located at $z=0$. 
Since $W(z)$ does not have the pole at infinity, $W(z)$ must be polynomial in $z^{-1}$. The latter fact, in turn, implies that $W^*(z)$ must be a polynomial matrix.
\end{proof}

\section{A numerical example}\label{sec:numerical-example}
In this section, we will show an application to stochastic realization of the algorithm used in the constructive proof of Theorem \ref{thmsf-dt-g}.
To this aim, let us consider a purely non-deterministic, second order process $\{y(t)\}_{t\in\Z}$ whose spectral density is
\[
\Phi(z) =\begin{bmatrix}
\frac{-2z+6-2z^{-1}}{-2z +5 -2z^{-1}} & z-1 & z-1 \\ z^{-1} -1 & -z +2 -z^{-1} &-z +2 -z^{-1}\\
z^{-1} -1 & -z +2 -z^{-1} &-z +2 -z^{-1}
\end{bmatrix}.
\] 
We want to compute a stochastically minimal, anti-causal realization of $\{y(t)\}_{t\in\Z}$ having all its zeroes in the (closed) unit disk. Since our method has been developed to compute a spectral factorization in the form
$\Phi(z)=W^\top (z^{-1}) W(z)$, this requirement corresponds to the choice $\mathscr{A}_z:=\{\,z\in\C\,:\, |z|<1\,\}$ and $\mathscr{A}_p:=\{\,z\in\C\,:\, |z|>1\,\}$.
Notice that $\Phi(z)$ is non-proper, it features a zero on the unit circle and it is rank deficient, namely $\mathrm{rk}(\Phi)=2$.

We now apply step-by-step the proposed factorization algorithm in order to compute a spectral factor $W(z)\in\mathbb{R}(z)^{2\times 3}$ analytic in $\mathscr{A}_p$ with right inverse analytic in $\mathscr{A}_z$.

{\em Step 1.} The Smith-McMillan canonical form of $\Phi(z)$ is given by
\[
D(z)=\diag\left[ \frac{1}{z(z-2)(z-\frac{1}{2})}, z(z-1)^2\right],
\]
$\Phi(z)$ can be decomposed as 
\[
\Phi(z)=C(z)D(z)F(z),
\] where $C(z)\in\mathbb{R}[z]^{3\times 2}$ and $F(z)\in\mathbb{R}[z]^{2\times 3}$ are unimodular matrices. 

{\em Step 2.} The matrices $\Lambda(z)$, $\Theta(z)$ and $\Sigma(z)$ defined in (\ref{eq:Dz-dec-thm1}) have the form
\begin{align*}
&\Lambda(z)=\diag\left[ \frac{1}{z\left(z-\frac{1}{2}\right)}, 1 \right],  \Theta(z)=\diag\left[1, z-1\right],\\ &\Sigma(z)=\diag\left[-\frac{1}{2 z^2}, -z^2\right].
\end{align*}
Note that $\Lambda(z)$ is analytic in $\mathscr{A}_p\setminus\{\infty\}$ with inverse analytic in $\mathscr{A}_z$. 
Let $A(z) =  C(z)\Sigma(z)\Lambda^*(z)$, $B(z) =  \Lambda(z)F(z)$.

{\em Step 3.} The matrix $\Psi(z)=\Theta(z)^{-1}N(z)\Theta(z)$, with $N(z)=A^*(z)B^{-R}(z)$, is given by
\begin{align*}
&\Psi(z) = \Theta(z)^{-1}N(z)\Theta(z)\\
&=\left[
\begin{smallmatrix}
 -\frac{1}{2}z+\frac{3}{2}-\frac{1}{2}z^{-1} & -\frac{9}{4}z^3+\frac{25}{2}z^2-\frac{43}{2}z+\frac{43}{4}+\frac{1}{2}z^{-1} \\
\frac{1}{2}z+\frac{43}{4}-\frac{43}{2}z^{-1}+\frac{25}{2}z^{-2} -\frac{9}{4}z^{-3}& \psi_{22}(z)
\end{smallmatrix}
\right].
\end{align*}
where $\psi_{22}(z):=\frac{9}{4} z^3+\frac{341}{8} z^2-\frac{1747}{8} z+\frac{2780}{8}-\frac{1747}{8} z^{-1}+\frac{341}{8}z^{-2} +\frac{9}{4}z^{-3}$.
It is worth noting that $\Psi(z)$ is para-Hermitian, L-unimodular and positive definite upon the unit circle.

{\em Step 4.}  Let $\Psi_1(z):=\Psi(z)$. The highest-column-degree coefficient matrix of $\Psi_1(z)$ is 
\[
\Psi_1^{\mathrm{hc}}= \left[
\begin{array}{cc}
 -\frac{1}{2} & -\frac{9}{4} \\
 \frac{1}{2} & \frac{9}{4} \\
\end{array}
\right].
\]
Since $\Psi_1^{\mathrm{hc}}$ is singular, we calculate a nonzero vector $\mathbf{v}_1\in\ker \Psi_1^{\mathrm{hc}}$. One such a vector is given, for instance, by $\mathbf{v}_1 = [9\ -2]^\top$. The highest maximum-degree active index set is $\mathcal{M}_1=\{2\}$, we construct the unimodular matrix $\Omega_1^{-1}(z)$ of the form (\ref{eq:matrix-reduction-dt})
\[
\Omega_1^{-1}(z) = \left[
\begin{array}{cc}
 1 & -\frac{9 }{2}z^2 \\
 0 & 1 \\
\end{array}
\right]
\]
in order to reduce the maximum degree of the second column of $\Psi_1(z)$,
\begin{align*}
&\Psi_2(z) = \Omega_1^{-*}(z)\Psi_1(z)\Omega_1^{-1}(z)\\ &=\left[
\begin{smallmatrix}
 -\frac{1}{2}z +\frac{3}{2}-\frac{1}{2}z^{-1} & \frac{23}{4} z^2-\frac{77}{4} z+\frac{43}{4}+\frac{1}{2}z^{-1} \\
 \frac{1}{2} z +\frac{43}{4}-\frac{77}{4} z^{-1}+\frac{23}{4} z^{-2} & -\frac{23}{4} z^2-\frac{973}{4} z+\frac{2123}{4} -\frac{973}{4} z^{-1}-\frac{23}{4}z^{-2} \\
\end{smallmatrix}\right].
\end{align*}
Since $\Psi_2^{\mathrm{hc}}$ is singular, we repeat the previous step. In this case, we have $\mathbf{v}_2=[23 \ 2]^\top\in\ker \Psi_2^{\mathrm{hc}}$, $\mathcal{M}_2=\{2\}$, and
\[
\Omega_2^{-1}(z) = \left[
\begin{array}{cc}
 1 & \frac{23}{2}z \\
 0 & 1 \\
\end{array}
\right].
\]
Hence, we compute the reduced matrix
\begin{align*}
\Psi_3(z) &= \Omega_2^{-*}(z)\Psi_2(z)\Omega_2^{-1}(z)\\ &=\left[
\begin{array}{cc}
-\frac{1}{2}z+\frac{3}{2}-\frac{1}{2}z^{-1} & -2 z+5+\frac{1}{2}z^{-1} \\
 \frac{1}{2}z+5-2z^{-1} & 2 z+21+ 2z^{-1} \\
\end{array}
\right].
\end{align*}
Actually, $\Psi_3^{\mathrm{hc}}$ is singular. In this case,  $\mathbf{v}_3=[-4 \ 1]^\top\in\ker \Psi_3^{\mathrm{hc}}$, $\mathcal{M}_3=\{2\}$,
\[
\Omega_3^{-1}(z)=\left[
\begin{array}{cc}
 1 & -4 \\
 0 & 1 \\
\end{array}
\right]
\]
and we obtain
\begin{align*}
\Psi_4(z) &= \Omega_3^{-*}(z)\Psi_3(z)\Omega_3^{-1}(z)\\
&=\left[
\begin{array}{cc}
 -\frac{1}{2}z+\frac{3}{2}-\frac{1}{2}z^{-1} & -1+\frac{5}{2 }z^{-1} \\
 \frac{5}{2}z-1 & 5 \\
\end{array}
\right].
\end{align*}
Yet another iteration is required; indeed $\Psi_4^{\mathrm{hc}}$ is singular. Thus we proceed by computing $\mathbf{v}_4=[-2 \ 1]^\top\in\ker \Psi_4^{\mathrm{hc}}$, $\mathcal{M}_3=\{1\}$,
\[
\Omega_4^{-1}(z)=\left[\begin{array}{cc}
 1 & 0 \\
 -\frac{1}{2}z & 1 \\
\end{array}
\right]
\]
and eventually we arrive at
\begin{align*}
\Psi_5 = \Omega_4^{-*}(z)\Psi_4(z)\Omega_4^{-1}(z)=\left[\begin{array}{cc}
 \frac{1}{4} & -1 \\
 -1 & 5 \\
\end{array}
\right].
\end{align*}
The latter matrix is constant and positive definite; therefore it admits a Cholesky factorization 
\[
\Psi_5=C^\top C,\quad C=\left[
\begin{array}{cc}
 \frac{1}{2} & -2 \\
 0 & 1 \\
\end{array}
\right].
\]
The fourth step of the algorithm is concluded, since we found a factorization $\Psi(z)=P^*(z)P(z)$, with $P(z)$ unimodular of the form
\begin{align*}
P(z)&=C\Omega_4(z)\Omega_3(z)\Omega_2(z)\Omega_1(z)\\
&=\left[
\begin{array}{cc}
 -z+\frac{1}{2} & -\frac{1}{4} z \left(18 z^2-55 z+39\right) \\
 \frac{1}{2}z & \frac{1}{4} \left(9 z^3-23 z^2+8 z+4\right) \\
\end{array}
\right].
\end{align*}

Finally, we have that 
\[
W(z)=P(z)\Theta(z)B(z)=\left[
\begin{array}{ccc}
 -\frac{1}{z} & \frac{1}{z}-1 & \frac{1}{z}-1 \\
 \frac{1}{2 z-1} & 0 & 0 \\
\end{array}
\right].
\]
is a stochastically minimal spectral factor of $\Phi(z)$ analytic in $\mathscr{A}_p$ with right inverse analytic in $\mathscr{A}_z$.

Therefore the sought for realization is
$$
y(t)=W^\top(z^{-1}) e(t)
$$
with $e(t)$ being white noise.

\section{Concluding remarks and future directions}\label{sec:conclusions}

In this paper we have established a general result on spectral factorization for an arbitrary discrete-time spectrum.
This result opens the way for many applications and generalizations of known results in several fields of systems theory such as estimation and stochastic realization.
In particular, for these applications it will be important to further investigate the links between arbitrary spectral factors and stochastic minimality.
A conjecture in this direction, which is currently under investigation, is the following.
\begin{conjecture}
Let $\Phi(z)\in\R(z)^{n\times n}$ be a spectrum of normal rank $\mathrm{rk}(\Phi)=r\neq 0$.
Let $\mathscr{A}_p$ and $\mathscr{A}_z$ be two unmixed-symplectic sets.
Let  $W(z)$ be a spectral factor satisfying points \ref{item:thmsf-dt-g(1)}), \ref{item:thmsf-dt-g(2)}) and \ref{item:thmsf-dt-g(3)}) of Theorem \ref{thmsf-dt-g}.
Then $W(z)$
is unique up to a constant, orthogonal matrix multiplier on the left, i.e., if $W_1(z)$ also satisfies points \ref{item:thmsf-dt-g(1)}), \ref{item:thmsf-dt-g(2)}) and \ref{item:thmsf-dt-g(3)}), then  $W_1(z)=TW(z)$ where $T\in\R^{r\times r}$ is orthogonal. 
\end{conjecture}

This conjecture would be a first step towards a complete parametrization of the set of all stochastically minimal right invertible spectral factors. We believe that this set can be parametrized very efficiently in terms of the all-pass divisors of a generalized {\em phase function} $T_0(z)$:\footnote{Notice that the definition of phase function employed in the conjecture is dual with respect to the classical definition used in stochastic realization,
\cite{Lindquist-P-85-siam,Lindquist-P-91-jmsec}.}
 
\begin{conjecture}
Let $\Phi(z)\in\R(z)^{n\times n}$ be a spectrum of normal rank $\mathrm{rk}(\Phi)=r\neq 0$.
Let $W_-(z)$ be the spectral factor corresponding to Theorem  \ref{thmsf-dt} and $\overline{W}_+(z)$ be the spectral factor corresponding to $\mathscr{A}_p=\mathscr{A}_z:=\{\, z\in\C\,:\, |z|<1\,\}$.
Let $T_0$ be the all-pass function defined by $T_0(z):=\overline{W}_+(z)W_-^{-R}(z)$.
Then, the set of all minimal right invertible spectral factors of $\Phi(z)$ is given by
\beann
&&\big\{\,W(z)=T_1(z) W_-(z)\,:\,\\
&&\hspace{1.5cm} T_1^\ast (z)T_1(z) =T_1 (z)T_1^\ast (z) =I_r,\\  
&&\hspace{1.5cm}  \delta_M(T_1(z))+\delta_M(T_0(z)T_1^\ast (z))=\delta_M(T_0(z))\,\big\}.
\eeann
\end{conjecture}

\end{document}